\documentclass[11pt,letterpaper]{amsart}
\usepackage[plainpages=false]{hyperref}
\usepackage{amsfonts,latexsym,rawfonts,amsmath,amssymb,amsthm, mathrsfs, lscape, calligra}
\usepackage{MnSymbol}
\usepackage{verbatim}
\usepackage{enumerate}
\usepackage{centernot}
\usepackage{mathtools}
\usepackage{tikz}
\usepackage{tikz-cd}

\usepackage[all]{xy}
\usepackage{graphicx,psfrag}

\usepackage{array, tabularx, color}

\usepackage{setspace}

\newtheorem{thm}{Theorem}[section]
\newtheorem{cor}[thm]{Corollary}
\newtheorem{lem}[thm]{Lemma}

\newtheorem{prop}[thm]{Proposition}

\newtheorem{exam}[thm]{Example}

\theoremstyle{remark}
\newtheorem{rmk}[thm]{Remark}

\theoremstyle{definition}
\newtheorem{defi}[thm]{Definition}

\newcommand{\CBbb}{\mathbb C}

\newcommand{\PBbb}{\mathbb P}

\newcommand{\ZBbb}{\mathbb Z}

\newcommand{\Ccal}{\mathcal C}

\newcommand{\Ecal}{\mathcal E}
\newcommand{\Fcal}{\mathcal F}
\newcommand{\Gcal}{\mathcal G}
\newcommand{\Hcal}{\mathcal H}
\newcommand{\Ical}{\mathcal I}

\newcommand{\Kcal}{\mathcal K}

\newcommand{\Ocal}{\mathcal O}

\newcommand{\Qcal}{\mathcal Q}

\newcommand{\Tcal}{\mathcal T}
\newcommand{\Ucal}{\mathcal U}

\DeclareMathOperator{\id}{id}

\DeclareMathOperator{\rank}{rank}

\DeclareMathOperator{\Spec}{Spec}

\DeclareMathOperator{\Ker}{Ker}

\newcommand{\sing}{{\rm Sing}}

\numberwithin{equation}{section}

\makeatletter
\@namedef{subjclassname@2020}{\textup{2020} Mathematics Subject Classification}
\makeatother

\begin{document}
\title[Point Singularities and Local $c_3$]{Point Singularities and Local Third Chern Classes for Rank-Two Torsion-free Sheaves on Threefolds}
\author[Chen]{Xuemiao Chen}
\address{Department of pure mathematics, University of Waterloo, Ontario, Canada, N2L 3G1}\email{x67Chen@uwaterloo.ca}
\thispagestyle{empty}

\begin{abstract}
In this paper, motivated by singularity formation in gauge theory, we study
the local third Chern class contribution carried by isolated point
singularities of rank-two torsion-free sheaves on complex threefolds. In the
local rank-two setting considered here, the invariant is defined in terms of
finite-length local algebraic data at the singular point. We prove that it can
be computed from data on the total family; in particular, it is deformation
invariant. We also prove that its parity recovers a
topological invariant of the underlying smooth complex rank-two vector bundle
on the boundary sphere.  We then give a relative $K$-theoretic interpretation:
a self-dual complex naturally associated with the sheaf defines a local
$K$-theoretic charge, and this charge is equal to the local third Chern
class.  For rank-two reflexive sheaves, we relate the same invariant to several
classical algebraic quantities, including the Fitting scheme and the
Buchsbaum--Rim multiplicity.  We also discuss applications to the boundary of
moduli spaces of Hermitian--Yang--Mills connections.
\end{abstract}

\maketitle

\tableofcontents

\bibliographystyle{amsplain}

\section{Introduction}

The purpose of this paper is to study the third Chern class contribution of rank-two torsion-free sheaves with isolated point singularities on complex threefolds. The starting point is the local term appearing in the bubbling analysis of degenerating rank-two bundles \cite{Chen26a}. We show that, in the rank-two isolated point setting, this term is an intrinsic local integer with deformation-theoretic, topological, and relative $K$-theoretic interpretations.

We begin with two motivations.

\begin{enumerate}
\item The first motivation comes from singularity formation in gauge theory. In \cite{Chen26a}, we established an algebraic bubbling identity for a family of rank-two bundles degenerating to a rank-two torsion-free sheaf with an isolated point singularity. The identity identifies the point contribution with a finite-length local algebraic quantity at the singularity. This is the reason for treating torsion-free sheaves with isolated point singularities on threefolds as local objects in their own right: they are not only technical degenerations of bundles, but natural local models for point-singular boundary phenomena in moduli spaces.

The torsion-free setting is essential here. Non-locally-free rank-two reflexive sheaves with isolated point singularities do not arise as such limits in this setting, while rank-two torsion-free sheaves with isolated point singularities may arise in this way. For explicit local examples, see \cite[Sections 3--5]{Chen26a}. From the analytic side, such point singularities are a basic difficulty in understanding compactifications of moduli spaces of Hermitian--Yang--Mills connections. To analyze these boundary phenomena, one needs local control of the algebraic quantities attached to the singularity and, in particular, an understanding of how they behave under deformations. One purpose of this paper is to isolate the rank-two features behind the bubbling identity and to study the resulting local invariant systematically.

Concretely, for a rank-two torsion-free sheaf $\Fcal$ with isolated singularities over a ball in $\CBbb^3$, the invariant at a singular point is defined by
$$
c_3(\Fcal,x)
=
\ell\bigl(\Ecal xt^1(\Fcal,\Ocal_B),x\bigr)
-
2\ell\bigl(\Fcal^{**}/\Fcal,x\bigr).
$$

\item The second motivation comes from the global theory of rank-two reflexive sheaves and from local Chern class theory. For a rank-two reflexive sheaf on three-dimensional projective space, the third Chern class is an important invariant in moduli problems; in particular, it controls the number of essential singularities \cite{Hartshorne:1980}. The classical formula expressing this invariant in terms of local Ext-lengths suggests that, in rank two, the point contribution should be governed by finite-length local algebra.

There are also general local Chern class formalisms behind this point of view. Iversen's theory gives local Chern classes for finite complexes of vector bundles which are exact away from a closed subset \cite{Iversen:1976}; in algebraic geometry, localized Chern classes are available in Fulton's intersection-theoretic framework \cite{Fulton:84}. Our goal here is complementary: we isolate the rank-two point-singularity case and compute the relevant local class directly.

In this paper, we therefore give a local treatment over a ball in $\CBbb^3$. The rank-two point-singularity setting has the special feature that the invariant can first be defined by finite-length local algebra, before passing to an associated complex which is exact away from $0$. This description makes several features of the invariant explicit: its flat-family formula, its deformation invariance, its parity, and its relation to boundary topology and to algebraic invariants such as the Fitting scheme and the Buchsbaum--Rim multiplicity.

Thus the treatment below gives a concrete rank-two local model for the general theory of local Chern classes. It shows what the local class measures in this singularity setting and explains why the resulting relative $K$-theoretic charge is governed by the same finite-length algebraic data. From the analytic point of view, this gives concrete local formulas for the point singularities which occur at the boundary of moduli spaces of Hermitian--Yang--Mills connections modulo gauge over compact K\"ahler threefolds, especially in relation to the Donaldson--Thomas--Segal program \cite{DonaldsonThomas:98, DonaldsonSegal:11}.
\end{enumerate}

We now describe the organization of the paper.

In Section \ref{Local Third Chern Class}, we introduce the local third Chern class for rank-two torsion-free sheaves with isolated point singularities on the unit ball in $\CBbb^3$. We prove the main deformation invariance theorem by showing that the invariant can be computed from data on the total family. The key step is first to establish the result for reflexive families and then to deduce the general torsion-free case. This phenomenon is special to rank two: in higher rank, examples show that no analogous simple statement holds. A crucial input is the self-duality of rank-two reflexive sheaves. As applications, we obtain several obstructions to realizing rank-two torsion-free sheaves with isolated point singularities as limits of locally free sheaves, or more generally of reflexive sheaves.

In Section \ref{Section-Parity}, we study the parity of the local third Chern class. To a rank-two torsion-free sheaf with an isolated point singularity, we associate a local topological invariant of the underlying smooth rank-two bundle on a small boundary sphere. This invariant detects the topological nontriviality of the bundle. We prove that it is recovered by the parity of the local third Chern class. The proof proceeds by perturbing isolated singularities to simple singularities, where the relevant invariants can be computed explicitly, and then applying deformation invariance and additivity. As a consequence, sheaves of odd parity cannot occur as isolated point singularities on the boundary of the moduli space of Hermitian--Yang--Mills connections modulo gauge. We also compute the topological invariant when a locally free extension exists. In particular, in the odd-parity case, it can be detected from the tangent cone of any admissible Hermitian--Yang--Mills connection on the sheaf.

In Section \ref{Section-K group}, we give a relative $K$-theoretic interpretation of the local third Chern class. After recalling the relevant classical relative $K$-groups and their relation with the Grothendieck group of finite-length modules, we construct, from a rank-two torsion-free sheaf, a self-dual complex which is exact away from the singular point. This complex defines a relative $K$-class on the pair $(B,S^5)$; the corresponding relative $K$-group is isomorphic to $\ZBbb$ and is generated by the Bott class. Using these identifications, we define a local $K$-theory charge at the singularity and prove that this charge is equal to the local third Chern class. In particular, in families of rank-two torsion-free sheaves with isolated point singularities, the resulting local $K$-theory charge is constant and can be computed algebraically from the total family.

In Section \ref{Section-Reflexive}, we specialize to rank-two reflexive sheaves. Since the quotient $\Fcal^{**}/\Fcal$ vanishes in the reflexive case, the local third Chern class is governed entirely by the local Ext sheaf $\Ecal xt^1(\Fcal,\Ocal_B)$. We show that this integer has several classical descriptions. Starting from a local presentation, we define a topological degree of the presentation map and prove that it agrees with the local third Chern class. We then identify the same integer with the length of the Fitting scheme and with the Buchsbaum--Rim multiplicity of the transpose presentation. Thus, in the reflexive case, the local invariant studied in this paper is recovered by standard invariants from topology and local algebra.

\subsection*{Notations and conventions}

\begin{itemize}
\item We work either on the three-dimensional unit ball $B\subset \CBbb^3$ centered at the origin, or on $X=B\times \Delta$, where $\Delta\subset \CBbb$ is the unit disk centered at the origin. For $0<r<1$, we write $B(r)\subset B$ for the concentric ball of radius $r$. We denote by $p:X\to \Delta$ the projection. A \emph{family} means a coherent sheaf on $X$ which is flat over $\Delta$ in our setting. If flatness is not assumed, we will simply say a coherent sheaf on $X$. A coherent sheaf on $B$ will be called a \emph{sheaf}. For $t\in \Delta$, we set $B_t=B\times \{t\}$ and identify $B_t$ with $B$ when no confusion can arise. If $\Ecal$ is a family, we write $\Ecal_t=\Ecal|_{B_t}$. For the central fiber, we often write $\Ccal=\Ecal|_{B_0}$.

\item We usually use $\Fcal$ to denote a sheaf on $B$, $\Ecal$ to denote a family on $X$, and $\Ccal$ to denote the central fiber of a family.

\item Unless otherwise specified, $\Ocal$ denotes the structure sheaf of the ambient space. When the ambient space is ambiguous, we write $\Ocal_B$ or $\Ocal_X$ explicitly. Duals and double duals are always taken on the space on which the sheaf lives. Thus $\Ecal^{**}$ denotes the double dual on $X$, while $\Ecal_t^{**}$ denotes the double dual of $\Ecal_t$ on $B_t$. We write $\Ecal xt^i(-,\Ocal)$ for sheaf Ext.

\item If $\tau$ is a finite-length sheaf, then, unless otherwise specified, we write
$$
\ell(\tau,x)=\operatorname{length}_{\Ocal_x}(\tau_x)=\dim_{\CBbb}\tau_x.
$$
If $\tau$ is supported at finitely many points, we write $\ell(\tau)=\sum_x\ell(\tau,x)$.

\item For a coherent $\Ocal_\Delta$-module $M$, the symbol $\operatorname{rank} M$ denotes its generic rank, or equivalently the rank of its locally free part.

\item Let $\iota_t:B_t\hookrightarrow X$ be the natural inclusion. For the central fiber, we write $\iota=\iota_0$. Thus, for a sheaf $\Ccal$ on $B_0$, the symbol $\iota_*\Ccal$ denotes its push-forward to $X$.
\end{itemize}

\subsection*{Acknowledgement}
The author thanks Richard Thomas and Jiahao Hu for helpful comments related to local third
Chern classes. He also thanks Song Sun for helpful
discussions related to this work. This work was
partially supported by NSERC and the ECR Supplement.

\section{Local third Chern classes}\label{Local Third Chern Class}

\subsection{Main results}
\begin{defi}\label{defi-2.1}
Let $\Fcal$ be a rank-two torsion-free sheaf on $B$ with finitely many isolated singularities. The local third Chern class of $\Fcal$ at a singular point $x$ is defined by
$$
c_3(\Fcal,x)=\ell(\Ecal xt^1(\Fcal,\Ocal_B),x)-2\ell(\Fcal^{**}/\Fcal,x).
$$
Here $\Fcal^{**}$ denotes the double dual of $\Fcal$ on $B$. The total third Chern class of $\Fcal$ over $B$ is defined by
$$
c_3(\Fcal)=\sum_{x\in \sing(\Fcal)} c_3(\Fcal,x).
$$
\end{defi}

We first record a flatness criterion for the sheaves considered below.

\begin{lem}\label{lem-flatness}
Let $\Ecal$ be a coherent sheaf on $X$. Suppose that, for every $t\in \Delta$, the fiber $\Ecal_t$ is a rank-two torsion-free sheaf on $B_t$ with isolated point singularities. Then $\Ecal$ is flat over $\Delta$.
\end{lem}

\begin{proof}
It suffices to show that $\Ecal$ has no torsion over the base. Fix $t_0\in \Delta$ and let $s=t-t_0$ be a local parameter at $t_0$. Let $\Tcal\subset \Ecal$ be the subsheaf killed by some power of $s$. Suppose $\Tcal\neq 0$. Then $\Tcal/s\Tcal\neq 0$ by Nakayama's lemma. Since $\Ecal/\Tcal$ has no $s$-torsion, the natural map
$$
\Tcal/s\Tcal\to \Ecal_{t_0}
$$
is injective. Moreover, $\Ecal/\Tcal$ is flat over $\Delta$ near $t_0$, since a torsion-free sheaf over a smooth curve is flat. Its nearby fibers have rank two, hence $(\Ecal/\Tcal)_{t_0}$ also has rank two. From the exact sequence
$$
0\to \Tcal/s\Tcal\to \Ecal_{t_0}\to (\Ecal/\Tcal)_{t_0}\to 0
$$
and the assumption that $\Ecal_{t_0}$ has rank two, it follows that $\Tcal/s\Tcal$ has rank zero on $B_{t_0}$. Thus $\Tcal/s\Tcal$ is a nonzero torsion subsheaf of the torsion-free sheaf $\Ecal_{t_0}$, a contradiction. Hence $\Tcal=0$ for every $t_0$, so $\Ecal$ has no base torsion. Therefore $\Ecal$ is flat over $\Delta$.
\end{proof}

We first give a family formula for the total local third Chern class.

\begin{thm}\label{Theorem-1}
Let $\Ecal$ be a family such that, for every $t\in \Delta$, the restriction $\Ecal_t$ is a rank-two torsion-free sheaf with finitely many isolated point singularities, all contained in $B(1/2)$. Then, for every $t\in \Delta$, the following holds
$$
c_3(\Ecal_t)=\rank(p_*\Ecal xt^1(\Ecal^{**},\Ocal_X))-2\rank(p_*(\Ecal^{**}/\Ecal)).
$$
\end{thm}

\begin{rmk}
\begin{enumerate}
\item Under the assumptions of Theorem \ref{Theorem-1}, the proof shows that $\Ecal^{**}/\Ecal$ is flat over $\Delta$. In particular, $p_*(\Ecal^{**}/\Ecal)$ is locally free.

\item Let $\Ecal$ be a family such that $\Ecal_t$ is locally free for $t\neq 0$, while the central fiber $\Ccal$ is a rank-two torsion-free sheaf with an isolated singularity at the origin. Then $\Ecal$ is reflexive. Therefore, by Theorem \ref{Theorem-1},
$$
c_3(\Ccal)=c_3(\Ecal_t)=0
$$
for $t\neq 0$. By Definition \ref{defi-2.1}, this is equivalent to
$$
\ell(\Ecal xt^1(\Ccal,\Ocal_B),0)=2\ell(\Ccal^{**}/\Ccal,0).
$$
In particular, this recovers the bubbling identity established in \cite{Chen26a}.

\item The key step in the proof is the case where $\Ecal$ is reflexive. Then the statement reads 
$$
c_3(\Ecal_t)=\rank (p_* \Ecal xt^1(\Ecal, \Ocal_X)).
$$
In this case, the proof below is a relative form of the local bubbling identity. In the single-point case, the bubbling contribution is measured by the $t$-torsion of $\Ecal xt^1(\Ecal,\Ocal_X)$ \cite{Chen26a}. For reflexive families, the natural object is instead the finite $\Ocal_\Delta$-module $p_*\Ecal xt^1(\Ecal,\Ocal_X)$. Its locally free part records the singularity charge that persists along nearby fibers, while its torsion part records the extra contribution appearing in the central fiber. The key calculation is that this torsion contribution is exactly cancelled by the reflexive-hull defect in the definition of $c_3$, leaving only the rank of $p_*\Ecal xt^1(\Ecal,\Ocal_X)$. The general case then follows from the reflexive case.

\item The rank two assumption in Theorem \ref{Theorem-1} is essential. In higher rank, the analogous deformation-invariance statement already fails for reflexive families. Indeed, on $X=B\times\Delta$, with coordinates $(z_1,z_2,z_3,t)$, consider the rank-three reflexive family
$$
0\to \Ocal_X
\xrightarrow{\left(\begin{smallmatrix}z_1\\ z_2\\ z_3\\ t\end{smallmatrix}\right)}
\Ocal_X^4\to \Ecal\to 0.
$$
For $t\neq 0$, the fiber $\Ecal_t$ is locally free, while $\Ccal=\Ecal|_{B_0}$ has a single isolated singularity at $0$. If one formally applies the expression in Definition \ref{defi-2.1} to this rank-three family, then
$$
c_3(\Ccal)=c_3(\Ccal,0)=1,\qquad c_3(\Ecal_t)=0\quad \text{for }t\neq 0.
$$
Thus the resulting quantity is not deformation invariant in higher rank.
   \item We are not aware of a previous explicit formulation of Theorem \ref{Theorem-1} in this local family form. The global relation between $c_3$ and $\Ecal xt^1$ for rank-two reflexive sheaves on projective three-space is classical \cite{Hartshorne:1980}; the point here is that the same type of relation admits a local formulation for families of rank-two torsion-free sheaves with isolated point singularities. As explained in Section \ref{Section-K group}, this can also be interpreted as conservation of a localized $K$-theoretic charge associated with the self-dual complex of the sheaf.
\end{enumerate}
\end{rmk}

We give the proof of Theorem \ref{Theorem-1} in the next section. Before doing so, we record some immediate applications. As a direct corollary of Theorem \ref{Theorem-1}, we obtain the following smoothability obstruction.

\begin{cor}
Let $\Fcal$ be a rank-two torsion-free sheaf on $B$ with an isolated singularity at the origin. If $c_3(\Fcal,0)\neq 0$, then $\Fcal$ cannot be smoothed in a family. More precisely, there is no family $\Ecal$ such that the central fiber $\Ccal$ satisfies $\Ccal\cong \Fcal$, while $\Ecal_t$ is locally free for $t\neq 0$.
\end{cor}

\begin{proof}
Suppose, to the contrary, that such a family $\Ecal$ exists. Since $\Ecal_t$ is locally free for $t\neq 0$, we have $c_3(\Ecal_t)=0$. By Theorem \ref{Theorem-1}, the total local third Chern class is invariant in the family. Hence
$$
c_3(\Fcal,0)=c_3(\Ccal)=c_3(\Ecal_t)=0,
$$
a contradiction.
\end{proof}

\begin{rmk}
For example, a rank-two reflexive sheaf with a nontrivial isolated singularity cannot be smoothed (\cite{Chen26a}).
\end{rmk}

As a second application, we obtain an obstruction to being a limit of rank-two reflexive sheaves.

\begin{cor}
Let $\Fcal$ be a rank-two torsion-free sheaf on $B$ with an isolated singularity at the origin. If $c_3(\Fcal,0)<0$, then $\Fcal$ cannot arise as the central fiber of a family of rank-two reflexive sheaves. More precisely, there is no family $\Ecal$ such that the central fiber $\Ccal$ satisfies $\Ccal\cong \Fcal$, while $\Ecal_t$ is reflexive for $t\neq 0$.
\end{cor}

\begin{proof}
Suppose, to the contrary, that such a family $\Ecal$ exists. Since $\Ecal_t$ is reflexive for $t\neq 0$, Definition \ref{defi-2.1} gives $c_3(\Ecal_t)\geq 0$. By Theorem \ref{Theorem-1}, the total local third Chern class is invariant in the family. Hence
$$
c_3(\Fcal,0)=c_3(\Ccal)=c_3(\Ecal_t)\geq 0,
$$
which contradicts the assumption that $c_3(\Fcal,0)<0$.
\end{proof}
\begin{rmk}
For example, let $\mathfrak m_0=(z_1,z_2,z_3)\subset\Ocal_B$ be the ideal sheaf of $0$, and let $\Fcal$ be defined by
$$
0\to \Ocal_B
\xrightarrow{\left(\begin{smallmatrix}z_1\\ z_2\\ z_3\end{smallmatrix}\right)}
\Ocal_B^2\oplus \mathfrak m_0
\to \Fcal \to 0,
$$
where the last component $z_3$ is regarded as a section of $\mathfrak m_0$. Then
$$
c_3(\Fcal,0)=-1.
$$
By the preceding corollary, $\Fcal$ cannot arise as the central fiber of a family whose nearby fibers are rank-two reflexive sheaves.

Similarly, if $\Fcal\subset \Ocal_B^2$ is a subsheaf such that $\Ocal_B^2/\Fcal$ is a nonzero finite-length sheaf, then $\Fcal^{**}\cong\Ocal_B^2$ and $\Ecal xt^1(\Fcal,\Ocal_B)=0$. Hence
$$
c_3(\Fcal)=-2\ell(\Ocal_B^2/\Fcal)<0,
$$
so $\Fcal$ cannot arise as such a limit.
\end{rmk}

As a final application, we record a constraint on the reflexive hull of the central fiber for the family we consider.

\begin{cor}
Let $\Ecal$ be a family satisfying the assumptions of Theorem \ref{Theorem-1}. Suppose that $\Ecal_t$ is reflexive for $t\neq 0$, while the central fiber $\Ccal$ has an isolated singularity at the origin. Then the singularity of $\Ccal$ is essential; equivalently, $\Ccal^{**}$ is not locally free at the origin.
\end{cor}

\begin{proof}
Suppose, to the contrary, that $\Ccal^{**}$ is locally free at the origin. Since $\Ccal$ is singular at the origin, the quotient $\Ccal^{**}/\Ccal$ has positive length at the origin. Moreover, $\Ecal xt^1(\Ccal,\Ocal_B)$ vanishes near the origin, because $\Ccal$ is a finite-length modification of the locally free sheaf $\Ccal^{**}$ on the smooth threefold $B$. Therefore Definition \ref{defi-2.1} gives
$$
c_3(\Ccal)=c_3(\Ccal,0)=-2\ell(\Ccal^{**}/\Ccal,0)<0.
$$
On the other hand, since $\Ecal_t$ is reflexive for $t\neq 0$, we have $c_3(\Ecal_t)\geq 0$. By Theorem \ref{Theorem-1},
$$
c_3(\Ccal)=c_3(\Ecal_t)\geq 0,
$$
a contradiction. Hence $\Ccal^{**}$ is not locally free at the origin.
\end{proof}

\subsection{Proof of Theorem \ref{Theorem-1}}
We first treat the reflexive case.

\begin{prop}\label{Prop-Reflexive-case}
Let $\Ecal$ be a rank-two reflexive family on $X$ such that, for every $t\in \Delta$, the sheaf $\Ecal_t$ has only isolated point singularities, all contained in $B(1/2)$. Then, for every $t\in \Delta$, the following holds
$$
c_3(\Ecal_t)=\rank(p_*\Ecal xt^1(\Ecal,\Ocal_X)).
$$
\end{prop}

We need the following identity.
\begin{lem}\label{Lemma-I1}
Suppose $\Fcal$ is a torsion-free sheaf on $B$ with an isolated point singularity at $x\in B$. Then
$$
\ell(\Ecal xt^2(\Fcal,\Ocal_B),x)=\ell(\Fcal^{**}/\Fcal,x).
$$
If, in addition, $\Fcal$ has rank two, then
$$
c_3(\Fcal,x)=\ell(\Ecal xt^1(\Fcal,\Ocal_B),x)-2\ell(\Ecal xt^2(\Fcal,\Ocal_B),x).
$$
\end{lem}

\begin{proof}
Applying $\Hcal om(\bullet,\Ocal_B)$ to the short exact sequence
$$
0\to \Fcal \to \Fcal^{**} \to \Fcal^{**}/\Fcal \to 0
$$
gives the exact sequence
$$
\Ecal xt^2(\Fcal^{**},\Ocal_B)\to \Ecal xt^2(\Fcal,\Ocal_B)\to \Ecal xt^3(\Fcal^{**}/\Fcal,\Ocal_B)\to \Ecal xt^3(\Fcal^{**},\Ocal_B).
$$
Since $\Fcal^{**}$ is reflexive on the smooth threefold $B$, $\Ecal xt^i(\Fcal^{**},\Ocal_B)=0$ for $i\geq 2$. Hence
$$
\Ecal xt^2(\Fcal,\Ocal_B)\cong \Ecal xt^3(\Fcal^{**}/\Fcal,\Ocal_B).
$$
By \cite[Lemma $2.10$]{Chen26a},
$$
\ell(\Ecal xt^3(\Fcal^{**}/\Fcal,\Ocal_B),x)=\ell(\Fcal^{**}/\Fcal,x).
$$
The first identity follows. The second identity follows from Definition \ref{defi-2.1}.
\end{proof}

We also record the following elementary observation.
\begin{lem}\label{Lemma-CokerAndKerSameDimension}
Suppose $\tau$ is a finite-length sheaf supported at $x$. Then
$$
\ell\bigl(\Kcal er(t:\tau\to\tau),x\bigr)=\ell\bigl(\Ccal oker(t:\tau\to\tau),x\bigr).
$$
\end{lem}

\begin{proof}
The endomorphism $t:\tau\to\tau$ gives exact sequences
$$
0\to \Kcal er(t)\to \tau \to \operatorname{Im}(t)\to 0
$$
and
$$
0\to \operatorname{Im}(t)\to \tau \to \Ccal oker(t)\to 0.
$$
Taking lengths and subtracting gives
$$
\ell\bigl(\Kcal er(t),x\bigr)=\ell\bigl(\Ccal oker(t),x\bigr).
$$
\end{proof}

Below we let $\Ecal$ be a reflexive family so that $\Ecal_t$ and the central fiber $\Ccal$ all have isolated point singularities.
\begin{lem}\label{Lemma-New Formula For C3}
The following holds
$$
\ell\bigl(\Kcal er(t:\tau\to\tau),x\bigr)=\ell(\Ccal^{**}/\Ccal,x).
$$
Here $\tau=\Ecal xt^1(\Ecal,\Ocal)$.  In particular,
$$
c_3(\Ccal,x)=\ell(\Ecal xt^1(\Ccal,\Ocal_B),x)-\ell(\Ecal xt^2(\Ccal,\Ocal_B),x)-\ell\bigl(\Kcal er(t:\tau\to\tau),x\bigr).
$$
\end{lem}

\begin{proof}
The argument is the local argument of \cite{Chen26a}; we recall it for completeness. We work near the point $x\in B_0\subset X$. After shrinking around $x$, the determinant of $\Ecal$ is trivial, and the rank-two reflexive identification $\Ecal^*\cong \Ecal\otimes(\det\Ecal)^{-1}$ gives $\Ecal^*\cong \Ecal$. Choose finitely many local generators of $\Ecal^*$ near $x$, giving an exact sequence
$$
0\to \Gcal \to \Ocal^n \to \Ecal^* \to 0.
$$
Applying $\Hcal om(\bullet,\Ocal)$ gives
$$
0\to \Ecal \to \Ocal^n \to \Gcal^* \to \Ecal xt^1(\Ecal^*,\Ocal) \to 0.
$$
Under the identification $\Ecal^*\cong \Ecal$, the last term is $\tau=\Ecal xt^1(\Ecal,\Ocal)$. Let $\Hcal$ denote the image sheaf of the morphism $\Ocal^n\to \Gcal^*$. Restricting the sequence $0\to \Ecal\to \Ocal^n\to \Hcal\to 0$ to the central fiber, and identifying $B_0$ with $B$, we get
$$
0\to \Ccal \to \Ocal_B^n \to \Hcal|_B \to 0.
$$
The torsion of the quotient $\Hcal|_B$ measures the saturation defect of $\Ccal$ in $\Ocal_B^n$, hence
$$
\ell(\Ccal^{**}/\Ccal,x)=\ell(\operatorname{tor}(\Hcal|_B),x),
$$
where $\operatorname{tor}(\Hcal|_B)$ denotes the torsion subsheaf of $\Hcal|_B$. On the other hand, from the exact sequence
$$
0\to \Hcal \to \Gcal^* \to \tau \to 0
$$
we obtain, after tensoring with $\Ocal_B$,
$$
0\to \Tcal or_1(\tau,\Ocal_B)\to \Hcal|_B\to (\Gcal^*)|_B\to \tau|_B\to 0.
$$
Since $\Gcal^*$ is reflexive, $\Gcal^*|_B$ is torsion-free near $x$. Therefore the torsion subsheaf of $\Hcal|_B$ is precisely $\Tcal or_1(\tau,\Ocal_B)$. Moreover, the exact sequence $0\to \Ocal\xrightarrow{t}\Ocal\to \Ocal_B\to 0$ gives
$$
\Tcal or_1(\tau,\Ocal_B)\cong \Kcal er(t:\tau\to\tau).
$$
Thus
$$
\ell(\Ccal^{**}/\Ccal,x)=\ell(\Kcal er(t:\tau\to\tau),x).
$$
Together with Lemma \ref{Lemma-I1}, this gives the stated formula for $c_3(\Ccal,x)$.
\end{proof}

We also have the following 
\begin{lem}\label{Lemma-Ext restriction}
    $\Ecal xt^k(\Ecal, \iota_* \Ocal_B) \cong \iota_*\Ecal xt^k(\Ccal, \Ocal_B)$  for any $k\geq 0.$
\end{lem}

\begin{proof}
The proof is identical to that of \cite[Lemma 2.11]{Chen26a}. Indeed, a local free resolution of $\Ecal$ restricts to a local free resolution of $\Ccal$. Applying $\Hcal om(-,\iota_*\Ocal_B)$ to the former resolution, and identifying the resulting complex with the pushforward of the complex obtained by applying the functor $\Hcal om_{\Ocal_B}(-,\Ocal_B)$ to the restricted resolution, gives
$$
\Ecal xt^k(\Ecal,\iota_*\Ocal_B)
\cong
\iota_*\Ecal xt^k(\Ccal,\Ocal_B)
$$
for all $k\geq 0$.
\end{proof}

\begin{lem}\label{Lemma-DVR-length}
Let $R=\Ocal_{\Delta,0}$ and let $M$ be a finite $R$-module. Then
$$
\dim_{\mathbb C}(M/tM)
=
\rank_R(M)+\dim_{\mathbb C}\Ker(t:M\to M).
$$
\end{lem}

\begin{proof}
Since $R$ is a DVR, we may write
$$
M\cong R^{\oplus r}\oplus T,
$$
where $T$ is a finite-length torsion $R$-module. Then
$$
\dim_{\mathbb C}(M/tM)=r+\dim_{\mathbb C}(T/tT).
$$
Multiplication by $t$ is injective on $R^{\oplus r}$, while for the
finite-length module $T$, Lemma \ref{Lemma-CokerAndKerSameDimension} gives
$$
\dim_{\mathbb C}(T/tT)
=
\dim_{\mathbb C}\Ker(t:T\to T).
$$
Hence
$$
\dim_{\mathbb C}(M/tM)
=
\rank_R(M)+\dim_{\mathbb C}\Ker(t:M\to M).
$$
\end{proof}

\begin{proof}[Proof of Proposition \ref{Prop-Reflexive-case}]
After translating the parameter, it suffices to prove the formula for the central fiber. Set $\tau_i=\Ecal xt^i(\Ecal,\Ocal)$ for $i=1,2$. Applying $\Hcal om(\Ecal,\bullet)$ to
$$
0\to \Ocal \xrightarrow{t} \Ocal \to \iota_*\Ocal_B \to 0
$$
gives the exact sequence
$$
\begin{aligned}
0\to \Kcal er(t:\tau_1\to\tau_1)\to \tau_1\xrightarrow{t}\tau_1
&\to \Ecal xt^1(\Ecal,\iota_*\Ocal_B)\to \tau_2\xrightarrow{t}\tau_2 \\
&\to \Ecal xt^2(\Ecal,\iota_*\Ocal_B)\to 0.
\end{aligned}
$$
Here we use that $\Ecal xt^3(\Ecal,\Ocal)=0$, since $\Ecal$ is reflexive. Moreover, $\tau_2$ has finite length. Since $p_*(\iota^*\tau_1)\cong (p_*\tau_1)/t(p_*\tau_1)$, Lemma \ref{Lemma-DVR-length} applied to $p_*\tau_1$ gives
$$
\ell(\iota^*\tau_1)=\ell\bigl(\Kcal er(t:\tau_1\to\tau_1)\bigr)+\rank(p_*\tau_1).
$$
Thus the exact sequence above gives
$$
\begin{aligned}
\ell\bigl(\Ecal xt^1(\Ecal,\iota_*\Ocal_B)\bigr)
={}&\ell\bigl(\Kcal er(t:\tau_1\to\tau_1)\bigr)+\rank(p_*\tau_1) \\
&+\ell\bigl(\Kcal er(t:\tau_2\to\tau_2)\bigr).
\end{aligned}
$$
Since $\tau_2$ has finite length, Lemma \ref{Lemma-CokerAndKerSameDimension} gives
$$
\ell\bigl(\Kcal er(t:\tau_2\to\tau_2)\bigr)=\ell\bigl(\Ecal xt^2(\Ecal,\iota_*\Ocal_B)\bigr).
$$
Therefore
$$
\begin{aligned}
&\ell\bigl(\Ecal xt^1(\Ecal,\iota_*\Ocal_B)\bigr)-\ell\bigl(\Kcal er(t:\tau_1\to\tau_1)\bigr)-\ell\bigl(\Ecal xt^2(\Ecal,\iota_*\Ocal_B)\bigr) \\
&\hspace{2cm}=\rank(p_*\tau_1).
\end{aligned}
$$
By Lemma \ref{Lemma-Ext restriction},
$$
\Ecal xt^k(\Ecal,\iota_*\Ocal_B)\cong \iota_*\Ecal xt^k(\Ccal,\Ocal_B)
$$
for all $k\geq 0$. Summing Lemma \ref{Lemma-New Formula For C3} over all singular points of $\Ccal$ identifies the left-hand side above with $c_3(\Ccal)$. Hence
$$
c_3(\Ccal)=\rank(p_*\Ecal xt^1(\Ecal,\Ocal)).
$$
The same argument after translating the parameter gives
$$
c_3(\Ecal_t)=\rank(p_*\Ecal xt^1(\Ecal,\Ocal))
$$
for every $t\in \Delta$. This proves the proposition.
\end{proof}

We now finish the proof of Theorem \ref{Theorem-1}.

\begin{proof}[Proof of Theorem \ref{Theorem-1}]
Set
$$
\Qcal=\Ecal^{**}/\Ecal.
$$
Since $\Ecal^{**}$ is reflexive on $X$, it is flat over $\Delta$. Hence restricting the short exact sequence
$$
0\to \Ecal \to \Ecal^{**} \to \Qcal \to 0
$$
to $B_t$ gives
$$
0\to \Tcal or_1^{\Ocal_X}(\Ocal_{B_t},\Qcal)\to \Ecal_t \to (\Ecal^{**})_t \to \Qcal_t \to 0.
$$
The sheaf $\Tcal or_1^{\Ocal_X}(\Ocal_{B_t},\Qcal)$ is supported at finitely many points on $B_t$. Since it is a subsheaf of the torsion-free sheaf $\Ecal_t$, it must vanish. Thus
$$
\Tcal or_1^{\Ocal_X}(\Ocal_{B_t},\Qcal)=0
$$
for every $t\in \Delta$. By the local criterion for flatness over the curve $\Delta$, the sheaf $\Qcal$ is flat over $\Delta$. Since $\Qcal$ is supported in $B(1/2)\times \Delta$, the sheaf $p_*\Qcal$ is locally free and
$$
\rank(p_*\Qcal)=\ell(\Qcal_t)
$$
for every $t\in \Delta$. Therefore
$$
0\to \Ecal_t \to (\Ecal^{**})_t \to \Qcal_t \to 0.
$$
Since $\Qcal_t$ has finite length, the sheaves $\Ecal_t$ and $(\Ecal^{**})_t$ have the same double dual and the same $\Ecal xt^1$. Hence
$$
c_3(\Ecal_t)=c_3((\Ecal^{**})_t)-2\ell(\Qcal_t)=c_3((\Ecal^{**})_t)-2\rank(p_*\Qcal).
$$
Applying Proposition \ref{Prop-Reflexive-case} to the reflexive family $\Ecal^{**}$ gives
$$
c_3((\Ecal^{**})_t)=\rank(p_*\Ecal xt^1(\Ecal^{**},\Ocal)).
$$
Combining the two identities gives
$$
c_3(\Ecal_t)=\rank(p_*\Ecal xt^1(\Ecal^{**},\Ocal))-2\rank(p_*(\Ecal^{**}/\Ecal)).
$$
This proves the theorem.
\end{proof}
\section{Topological interpretation of the parity of the local third Chern classes}\label{Section-Parity}
\subsection{Simple singularities and a topological invariant}
\begin{defi}
 Suppose $\Fcal$ is a rank-two torsion-free sheaf with isolated point singularities over the three-dimensional ball $B$. A point $x$ is called a simple singularity of $\Fcal$ if $\Fcal$ is reflexive at $x$ and
$$
\ell\bigl(\Ecal xt^1(\Fcal,\Ocal_B),x\bigr)=1.
$$
\end{defi}

\begin{lem}
Let $x$ be a simple singularity of a rank-two torsion-free sheaf $\Fcal$. Then, up to a local holomorphic change of coordinates centered at $x$, the sheaf $\Fcal$ is locally modeled on
$$
0\to \Ocal \xrightarrow{\begin{pmatrix}z_1\\ z_2\\ z_3\end{pmatrix}} \Ocal^3 \to \Fcal \to 0.
$$
\end{lem}

\begin{proof}
Work near $x$, and let $\mathfrak m_x$ be the maximal ideal at $x$. Since $\Fcal$ is reflexive of rank two, it has a minimal free resolution
$$
0\to \Ocal^r \xrightarrow{\Phi} \Ocal^{r+2}\to \Fcal\to 0,
$$
where every entry of $\Phi$ lies in $\mathfrak m_x$. Dualizing gives
$$
0\to \Fcal^* \to \Ocal^{r+2} \xrightarrow{\Phi^t} \Ocal^r \to \Ecal xt^1(\Fcal,\Ocal)\to 0.
$$
Tensoring with $\CBbb=\Ocal/\mathfrak m_x$ gives $\Ecal xt^1(\Fcal,\Ocal)\otimes \CBbb\cong \CBbb^r$. Since $x$ is a simple singularity, $\ell(\Ecal xt^1(\Fcal,\Ocal),x)=1$, hence $r=1$. Thus $\Phi=(f_1,f_2,f_3)^t$, and
$$
\Ecal xt^1(\Fcal,\Ocal)\cong \Ocal/(f_1,f_2,f_3).
$$
Again by simplicity, this quotient has length one, so $(f_1,f_2,f_3)=\mathfrak m_x$. Therefore the classes of $f_1,f_2,f_3$ form a basis of $\mathfrak m_x/\mathfrak m_x^2$. By the inverse function theorem, after a local holomorphic change of coordinates centered at $x$ we may take $f_i=z_i$ for $i=1,2,3$.
\end{proof}
\begin{defi}
Let $\Fcal$ be a rank-two torsion-free sheaf on $B$ with finitely many isolated point singularities. For each singular point $x$, choose a sufficiently small sphere $S_x^5$ centered at $x$, and define $\delta(\Fcal,x)\in \ZBbb/2\ZBbb$ by
$$
\delta(\Fcal,x)=
\begin{cases}
1,& \text{if the complex rank-two bundle $\Fcal|_{S_x^5}$ is topologically nontrivial},\\
0,& \text{otherwise.}
\end{cases}
$$
We define $\delta(\Fcal)$ similarly, using the restriction of $\Fcal$ near the boundary sphere $\partial B$.
\end{defi}

\begin{rmk}
Proposition \ref{AdditivePropertyOfTopologicalInvariants} below shows that this topological invariant is additive in the sense that
$$
\delta(\Fcal)=\sum_x \delta(\Fcal,x),
$$
where the sum is taken over the singular points of $\Fcal$.
\end{rmk}

\subsection{The parity of the local third Chern class}
Now we state the main result of this section.
\begin{thm}\label{Theorem-2}
Let $\Fcal$ be a rank-two torsion-free sheaf on $B$ with finitely many isolated point singularities. Then
$$
\delta(\Fcal)\equiv c_3(\Fcal)\pmod{2}.
$$
\end{thm}

The proof of this result will be given in the next section. Before this, we record the following obstruction to deforming a singularity in the topological setting.

\begin{cor}
Suppose $\Fcal$ is a rank-two torsion-free sheaf with a point singularity at $x$ and that $c_3(\Fcal, x)\equiv 1\pmod{2}$. Then, away from $x$, it cannot locally arise as the smooth limit of a sequence of holomorphic vector bundles over $B$. In particular, it cannot occur as an isolated singularity in the Uhlenbeck compactification of the moduli space of Hermitian--Yang--Mills connections modulo gauge.
\end{cor}

\begin{rmk}
There are nontrivial isolated singularities satisfying
$$
c_3(\Fcal,0)=0
$$
which arise as limits of sequences of holomorphic vector bundles on $B$;
see \cite{Chen26a}.
\end{rmk}

The parity of the local third Chern class can also be computed from a locally free extension, when such an extension exists. Suppose $\Fcal$ is a rank-two reflexive sheaf on $B$ with an isolated point singularity at $x$. Let $\pi:\widehat B=\mathrm{Bl}_xB\to B$ be the blow-up and let $D=\pi^{-1}(x)\cong\PBbb^2$. By a locally free extension of $\Fcal$, we mean a locally free sheaf $\widehat{\Fcal}$ on $\widehat B$ whose restriction to $\widehat B\setminus D$ is identified with the pullback of $\Fcal|_{B\setminus\{x\}}$.

\begin{cor}\label{Computation}
Suppose $\Fcal$ is a rank-two reflexive sheaf on $B$ with an isolated point singularity at $x$, and suppose that $\Fcal$ admits a locally free extension $\widehat{\Fcal}$ on $\widehat B=\mathrm{Bl}_xB$. Then
$$
c_3(\Fcal,x)\equiv
c_1(\widehat{\Fcal}|_D)c_2(\widehat{\Fcal}|_D)
\pmod{2}.
$$
\end{cor}

Combining Corollary \ref{Computation} with the main results of \cite{ChenSun:18, ChenSun:19}, we obtain the following consequence. The point is that the parity constraint is visible in the tangent-cone decomposition: it is carried by the cone connection when the link is stable, and by the bubbling cycle when it is not.

\begin{cor}\label{OddTangentCone}
Suppose $\Fcal$ is a rank-two reflexive sheaf on $B$ with a homogeneous isolated point singularity at $x$, and suppose that
$$
\ell(\Ecal xt^1(\Fcal,\Ocal_B),x)\equiv 1\pmod{2}.
$$
Let $\underline{\Fcal}$ be the link of $\Fcal$ at $x$. Then, for every admissible Hermitian--Yang--Mills metric $H$ on $\Fcal$, the tangent cone has odd total energy. More precisely, if $\underline{\Fcal}$ is stable, the tangent cone is the smooth cone connection induced by $\underline{\Fcal}$, and $c_1(\underline{\Fcal})c_2(\underline{\Fcal})\equiv 1\mod 2$. If $\underline{\Fcal}$ is not stable, the tangent cone consists of a flat cone connection together with a bubbling cycle of odd total multiplicity.
\end{cor}

\begin{proof}
Since $\Fcal$ is reflexive,
$$
c_3(\Fcal,x)=\ell(\Ecal xt^1(\Fcal,\Ocal_B),x)\equiv 1\pmod{2}.
$$
Applying Corollary \ref{Computation} to the blow-up extension whose restriction to the exceptional divisor is $\underline{\Fcal}$ gives
$$
c_1(\underline{\Fcal})c_2(\underline{\Fcal})\equiv 1\pmod{2}.
$$
By \cite{ChenSun:18, ChenSun:19}, the analytic tangent cone is determined by the Harder--Narasimhan--Seshadri graded sheaf of $\underline{\Fcal}$, and its total energy is the sum of the connection contribution and the bubbling-cycle contribution. If $\underline{\Fcal}$ is stable, the graded sheaf is $\underline{\Fcal}$ itself, so no bubbling occurs and the odd parity is carried by the smooth cone connection. If $\underline{\Fcal}$ is not stable, the graded double dual is a sum of line bundles, hence the connection part is flat, and the same odd parity is carried by the bubbling cycle. Thus the total energy of the tangent cone is odd.
\end{proof}

\subsection{Proof of Theorem \ref{Theorem-2}}
We begin with an elementary observation.
\begin{lem}\label{TopologicalInvairantsOfSimplesSingularities}
Suppose $\Fcal$ has a simple singularity at $x$. Then
$$
\delta(\Fcal,x)\equiv 1\pmod{2}.
$$
\end{lem}

\begin{proof}
By the previous lemma, after choosing local coordinates, we may assume that the underlying smooth bundle of $\Fcal$ restricted to a sufficiently small sphere centered at $x$ is the pullback of the tangent bundle of $\PBbb^2$. If this bundle were topologically trivial, then the tangent bundle of $S^5$ would also be trivial. But it is well known that among spheres, only $S^1$, $S^3$, and $S^7$ have trivial tangent bundle. This is a contradiction. The claim follows.
\end{proof}

We will use the following additivity property of the topological $\ZBbb/2\ZBbb$-invariant.

\begin{prop}\label{AdditivePropertyOfTopologicalInvariants}
Suppose $\Fcal$ is a rank-two torsion-free sheaf on $B$ with finitely many isolated point singularities $x_i$. Then
$$
\delta(\Fcal)=\sum_i\delta(\Fcal,x_i),
$$
where the sum is taken in $\ZBbb/2\ZBbb$.
\end{prop}

Before the proof, we record a two-point cancellation lemma.

\begin{lem}\label{TwoSimpleSingularity}
There exists a rank-two reflexive sheaf $\Fcal$ on $B$ with two simple singularities such that $\delta(\Fcal)=0$.
\end{lem}

\begin{proof}
On $B\times \Delta$, consider
$$
0\to \Ocal \xrightarrow{\begin{pmatrix}z_1\\ z_2\\ z_3(z_3-t)\end{pmatrix}}\Ocal^3\to \Ecal\to 0.
$$
For $t\neq 0$, the fiber $\Ecal_t$ has exactly two simple singularities. Let $\Ccal=\Ecal|_{B_0}$. Since the family is locally free near $\partial B$, $\delta(\Ecal_t)=\delta(\Ccal)$. By \cite{Chen26a}, there is a family $\Gcal$ such that $\Gcal_t$ is locally free for $t\neq 0$ and $(\Gcal|_{B_0})^{**}\cong \Ccal$. Hence $\delta(\Ccal)=0$, so $\delta(\Ecal_t)=0$ for $t\neq 0$.
\end{proof}

\begin{cor}\label{TrivialityOfTwo}
Suppose $F$ is a complex rank-two bundle on $B\setminus \{x_1,x_2\}$ which is topologically nontrivial near both $x_1$ and $x_2$. Then $F$ is topologically trivial near $\partial B$.
\end{cor}

\begin{proof}
The space $B\setminus \{x_1,x_2\}$ is homotopy equivalent to $S^5\vee S^5$. Hence the topological type of $F$ is determined by its two local classes. Since both classes are the nonzero element of $\ZBbb/2\ZBbb$, Lemma \ref{TwoSimpleSingularity} realizes the same local data with trivial boundary class. Therefore $F$ is trivial near $\partial B$.
\end{proof}

\begin{proof}[Proof of Proposition \ref{AdditivePropertyOfTopologicalInvariants}]
Let $F$ be the underlying topological vector bundle of $\Fcal$ away from its singularities. Fill in every topologically trivial singularity by a trivial bundle. It remains to count the topologically nontrivial singularities. We use strong induction on their number $n$. The cases $n=0,1$ are clear. If $n\geq 2$, choose two nontrivial singularities and a ball-like connected neighborhood $U$ containing them and no other singularities. By Corollary \ref{TrivialityOfTwo}, $F$ is trivial near $\partial U$, so we may glue in a trivial bundle over $U$. This removes two nontrivial singularities, and leaves the boundary class unchanged. The result follows by induction.
\end{proof}

We reduce the parity statement to simple singularities by the following perturbation result.

\begin{prop}\label{Perturbation}
Suppose $\Fcal$ is a rank-two reflexive sheaf on $B$ with finitely many isolated point singularities. Then there exists a flat family $\Ecal$ on $B\times \Delta$ such that $\Ccal\cong \Fcal$ and, for every $t\in \Delta\setminus\{0\}$, the sheaf $\Ecal_t$ has only simple singularities. Moreover, after shrinking $\Delta$, the family is locally free near $\partial B$.
\end{prop}

We first make a local reduction. After shrinking $B$, choose a presentation
$$
0\to \Ocal^r \xrightarrow{\Phi} \Ocal^{r+2}\to \Fcal\to 0.
$$
Let $M=\operatorname{Hom}(\CBbb^r,\CBbb^{r+2})$. We regard $\Phi$ as a holomorphic map
$$
\Phi:B\to M.
$$
Let
$$
\Sigma=\{A\in M:\rank(A)\leq r-1\}.
$$
Then $\Sigma$ is irreducible of codimension three. Let $\Ucal$ be the tautological rank-two reflexive sheaf on $M$, defined by
$$
0\to \Ocal_M^r \xrightarrow{\mathsf A} \Ocal_M^{r+2}\to \Ucal\to 0,
$$
where $\mathsf A|_A=A$ for $A\in M$.

\begin{lem}\label{TransversalImpliesSimple}
Suppose $\Phi(x)\in \Sigma_{\mathrm{sm}}$ and $\Phi$ is transverse to $\Sigma$ at $x$. Then $\Fcal$ has a simple singularity at $x$.
\end{lem}

\begin{proof}
After translating $x$ to $0$ and changing bases, assume
$$
\Phi(0)=
\begin{pmatrix}
\id_{r-1} & 0\\
0 & 0\\
0 & 0\\
0 & 0
\end{pmatrix}.
$$
Near this point, $\Sigma$ is cut out by the last three entries of the final column. By transversality, their pullbacks have independent differentials. Thus, after choosing local coordinates $(z_1,z_2,z_3)$ on $B$ and changing bases again,
$$
\Phi=
\begin{pmatrix}
\id_{r-1} & 0\\
0 & z_1\\
0 & z_2\\
0 & z_3
\end{pmatrix}.
$$
After cancelling the trivial summand, this is the simple model.
\end{proof}

\begin{rmk}
The converse also holds, but will not be used.
\end{rmk}

\begin{lem}
    There exists $\Phi'\in M$ such that for $0<|t|\ll 1$, $\Phi_t=\Phi+t\Phi'$ drops rank precisely at the smooth locus of $\Sigma$, and $\Phi_t$ is transverse to $\Sigma$ at those points.
\end{lem}

\begin{proof}
This is a standard consequence of the complex analytic version of Sard's theorem. Consider the map
$$
\Psi:B\times M \to M,\qquad \Psi(x,\xi)=\Phi(x)+\xi.
$$
Since varying $\xi$ simply translates the image, the map $\Psi$ is transverse to $\Sigma$. More precisely, it is transverse to each smooth locus in the stratification of $\Sigma$.   It follows that the set of parameters $\xi\in M$ for which the translated map
$$
\Psi_\xi(x):=\Psi(x,\xi)=\Phi(x)+\xi
$$
fails to be transverse to $\Sigma$ is contained in a proper analytic subvariety of $M$. Therefore, we may choose a smooth curve $\Delta\subset M$ passing through $0$ such that, for every sufficiently small nonzero $t\in \Delta$, the map
$$
\Psi_t(x):=\Psi(x,t)
$$
is transverse to $\Sigma$. This proves the claim.
\end{proof}

Fix such a family $\Phi_t$, write $\Phi_\bullet:B\times\Delta\to M$ for the induced map, and set $\Ecal=\Phi_\bullet^*\Ucal$. Then Proposition \ref{Perturbation} follows from the next lemma.

\begin{lem}
After shrinking $\Delta$, the sheaf $\Ecal$ is reflexive on $B\times \Delta$ and flat over $\Delta$. Moreover, $\Ccal\cong \Fcal$, for every $t\neq 0$ all singularities of $\Ecal_t$ are simple, and $\Ecal$ is locally free near $\partial B\times \Delta$.
\end{lem}

\begin{proof}
By construction,
$$
\Ecal_t=\Phi_t^*\Ucal.
$$
Thus $\Ccal\cong \Fcal$. For $t\neq 0$, the map $\Phi_t$ meets $\Sigma$ only along $\Sigma_{\mathrm{sm}}$ and is transverse there. Hence Lemma \ref{TransversalImpliesSimple} implies that all singularities of $\Ecal_t$ are simple.

It remains to prove reflexivity. The construction gives an exact sequence
$$
0\to \Ocal_{B\times \Delta}^r \xrightarrow{\Phi_\bullet} \Ocal_{B\times \Delta}^{r+2}\to \Ecal\to 0.
$$
The degeneracy locus of $\Phi_\bullet$ is finite over $\Delta$, hence has codimension at least three in $B\times \Delta$. Therefore $\Ecal$ is locally free in codimension two. Since $B\times \Delta$ is smooth and $\Ecal$ has projective dimension at most one, by the Auslander--Buchsbaum formula, $\Ecal$ satisfies $S_2$. Thus $\Ecal$ is reflexive. In particular, $\Ecal$ is torsion-free over the smooth curve $\Delta$, hence flat over $\Delta$.

Finally, since $\Fcal$ is locally free near $\partial B$, the map $\Phi$ avoids $\Sigma$ there. After shrinking $\Delta$, the same holds for $\Phi_t$, so $\Ecal$ is locally free near $\partial B\times \Delta$.
\end{proof}

\begin{proof}[Proof of Theorem \ref{Theorem-2}]
Let $\Fcal$ be a rank-two torsion-free sheaf with isolated point singularities contained in $B(1/2)$. The quotient $\Fcal^{**}/\Fcal$ has finite length and $\Ecal xt^1(\Fcal,\Ocal_B)\cong \Ecal xt^1(\Fcal^{**},\Ocal_B)$, so
$$
c_3(\Fcal)\equiv c_3(\Fcal^{**}) \pmod{2}.
$$
Moreover, $\Fcal$ and $\Fcal^{**}$ agree away from finitely many points, hence have the same $\delta$-invariant. Thus we may assume that $\Fcal$ is reflexive.

By Proposition \ref{Perturbation}, choose a flat family $\Ecal$ with $\Ccal\cong \Fcal$ such that, for $0<|t|\ll 1$, all singularities of $\Ecal_t$ are simple and contained in $B(1/2)$. By Theorem \ref{Theorem-1},
$$
c_3(\Ecal_t)=c_3(\Ccal)=c_3(\Fcal).
$$
Since the family is locally free near $\partial B\times \Delta$,
$$
\delta(\Ecal_t)=\delta(\Fcal).
$$
It remains to prove
$$
\delta(\Ecal_t)\equiv c_3(\Ecal_t)\pmod{2}.
$$
Let $x_i$ be the singular points of $\Ecal_t$. By Proposition \ref{AdditivePropertyOfTopologicalInvariants} and the definition of $c_3(\Ecal_t)$, it suffices to check
$$
\delta(\Ecal_t,x_i)\equiv c_3(\Ecal_t,x_i)\pmod{2}
$$
for each $i$. Since all $x_i$ are simple singularities, this follows from Lemma \ref{TopologicalInvairantsOfSimplesSingularities}.
\end{proof}

\subsection{Computations of the invariants}
The following elementary fact will be needed. 
\begin{lem}\label{Lemma-Homogeneous}
Suppose $\Fcal$ is homogeneous, i.e.
$$
\Fcal\cong j_*\rho^*\underline{\Fcal},
$$
where $\rho:B\setminus\{0\}\to \PBbb^2$ is the natural projection, $j:B\setminus\{0\}\hookrightarrow B$, and $\underline{\Fcal}$ is a rank-two holomorphic vector bundle on $\PBbb^2$. Then
$$
\ell\bigl(\Ecal xt^1(\Fcal,\Ocal_B),0\bigr)
=
\dim_{\CBbb}\left(\bigoplus_{k\in\ZBbb}H^1(\PBbb^2,\underline{\Fcal}(k))\right).
$$
\end{lem}

\begin{proof}
Let $q:\widehat B\to B$ be the blow-up of $B$ at $0$, let $D=q^{-1}(0)\cong\PBbb^2$ be the exceptional divisor, and let $\kappa:D\hookrightarrow \widehat B$ be the natural inclusion. Since $\Fcal$ is homogeneous, it is obtained from a locally free sheaf $\widehat{\Fcal}$ on $\widehat B$ whose restriction to the exceptional divisor is
$$
\widehat{\Fcal}|_D\cong \underline{\Fcal}.
$$
For $k\in\ZBbb$, set
$$
\widehat{\Fcal}(k):=\widehat{\Fcal}\otimes\Ocal_{\widehat B}(kD).
$$
With this convention, we have the standard exact sequence
$$
0\to \widehat{\Fcal}(k-1)\to \widehat{\Fcal}(k)\to \kappa_*\underline{\Fcal}(-k)\to 0.
$$
Applying $q_*$ to this short exact sequence gives the long exact sequence
$$
\begin{aligned}
0&\to q_*\widehat{\Fcal}(k-1)\to q_*\widehat{\Fcal}(k)\to H^0(\PBbb^2,\underline{\Fcal}(-k))\\
&\to R^1q_*\widehat{\Fcal}(k-1)\to R^1q_*\widehat{\Fcal}(k)\to H^1(\PBbb^2,\underline{\Fcal}(-k))\\
&\to R^2q_*\widehat{\Fcal}(k-1)\to R^2q_*\widehat{\Fcal}(k)\to H^2(\PBbb^2,\underline{\Fcal}(-k))\to 0.
\end{aligned}
$$
Since $\widehat{\Fcal}$ is homogeneous, this yields the exact sequence
$$
0\to R^1q_*\widehat{\Fcal}(k-1)\to R^1q_*\widehat{\Fcal}(k)\to H^1(\PBbb^2,\underline{\Fcal}(-k))\to 0.
$$
In particular, for $k\gg 0$, the sequence stabilizes:
$$
R^1q_*\widehat{\Fcal}(k)=R^1q_*\widehat{\Fcal}(k+1)=\cdots.
$$
It follows from the exact sequences above that, for $k\gg 0$,
$$
\ell\bigl(R^1q_*\widehat{\Fcal}(k),0\bigr)
=
\dim_{\CBbb}\left(\bigoplus_{j\in\ZBbb}H^1(\PBbb^2,\underline{\Fcal}(j))\right).
$$

It remains to identify the stable sheaf $R^1q_*\widehat{\Fcal}(k)$ with $\Ecal xt^1(\Fcal,\Ocal_B)$ for $k\gg0$. Choose a resolution of $\underline{\Fcal}^*$, which naturally induces a resolution of $\widehat{\Fcal}^*$ of the form
$$
0\to \Ocal_{\widehat B}(nD)^r
\to \Ocal_{\widehat B}(mD)^{r+2}
\to \widehat{\Fcal}^* \to 0.
$$
Dualizing, we obtain
$$
0\to \widehat{\Fcal}
\to \Ocal_{\widehat B}(-mD)^{r+2}
\to \Ocal_{\widehat B}(-nD)^r \to 0.
$$
Tensoring with $\Ocal_{\widehat B}(kD)$ for $k\gg0$, and then pushing forward, gives
$$
0\to \Fcal
\to \Ocal_B^{r+2}
\to \Ocal_B^r
\to R^1q_*\widehat{\Fcal}(k) \to 0.
$$
Here we have used the fact that
$$
q_*\widehat{\Fcal}(k)\cong \Fcal
$$
for $k\gg0$; see \cite[Lemma~2.7]{Chen2025}. On the other hand, taking double dual of this sequence yields
$$
0\to \Fcal
\to \Ocal_B^{r+2}
\to \Ocal_B^r
\to \Ecal xt^1(\Fcal^*,\Ocal_B) \to 0,
$$
where we used the fact that
$$
\Ecal xt^1(\Fcal^*,\Ocal_B)\cong \Ecal xt^1(\Gcal,\Ocal_B)
$$
for any $\Gcal\subset \Hcal$ with $\Hcal$ locally free and $\Hcal/\Gcal$ supported at $0$. Comparing the two exact sequences, we obtain
$$
R^1q_*\widehat{\Fcal}(k)\cong \Ecal xt^1(\Fcal^*,\Ocal_B)
$$
for $k\gg0$. After shrinking $B$ and trivializing $\det\Fcal$, the rank-two reflexive identification gives $\Fcal^*\cong\Fcal$. Thus
$$
R^1q_*\widehat{\Fcal}(k)\cong \Ecal xt^1(\Fcal,\Ocal_B).
$$
This completes the proof.
\end{proof}
We now record a few simple topological consequences of Theorem
\ref{Theorem-2}.

\begin{cor}\label{Proposition-Z2 invariants for homogenous singularity}
Suppose $\Fcal$ is locally homogeneous at $0$, that is,
$$
\Fcal=j_*\rho^*\underline{\Fcal}
$$
for some rank-two holomorphic vector bundle $\underline{\Fcal}$ on
$\PBbb^2$. Then
$$
\delta(\Fcal)\equiv
c_1(\underline{\Fcal})c_2(\underline{\Fcal})
\pmod{2}.
$$
\end{cor}

\begin{proof}
By Theorem \ref{Theorem-2}, we have
$$
\delta(\Fcal)\equiv \ell\bigl(\Ecal xt^1(\Fcal,\Ocal)\bigr)\pmod{2}.
$$
Thus it suffices to prove that
$$
\ell\bigl(\Ecal xt^1(\Fcal,\Ocal)\bigr)
\equiv
c_1(\underline{\Fcal})c_2(\underline{\Fcal})
\pmod{2}.
$$
After twisting by a line bundle, we may assume without loss of generality that
$$
c_1(\underline{\Fcal})\in \{0,-1\}.
$$
By Lemma \ref{Lemma-Homogeneous}, we have
$$
\ell(\Ecal xt^1(\Fcal,\Ocal))
=\dim_{\CBbb}
(\bigoplus_k H^1(\PBbb^2,\underline{\Fcal}(k))).
$$
On the other hand, Serre duality gives
$$
H^1(\PBbb^2,\underline{\Fcal}(k))
\cong
H^1(\PBbb^2,\underline{\Fcal}^*(-3-k))^*.
$$
If $c_1(\underline{\Fcal})=0$, then
$\underline{\Fcal}^*\cong \underline{\Fcal}$, and hence
$$
H^1(\PBbb^2,\underline{\Fcal}(k))
\cong
H^1(\PBbb^2,\underline{\Fcal}(-3-k))^*.
$$
The summands are therefore paired with no fixed term, so
$$
\ell\bigl(\Ecal xt^1(\Fcal,\Ocal)\bigr)\equiv 0\pmod{2}.
$$
This agrees with
$$
c_1(\underline{\Fcal})c_2(\underline{\Fcal})\equiv 0\pmod{2}.
$$
If $c_1(\underline{\Fcal})=-1$, then
$\underline{\Fcal}^*\cong \underline{\Fcal}(1)$, and hence
$$
H^1(\PBbb^2,\underline{\Fcal}(k))
\cong
H^1(\PBbb^2,\underline{\Fcal}^*(-3-k))^*
\cong
H^1(\PBbb^2,\underline{\Fcal}(-2-k))^*.
$$
Thus all summands are paired except the middle term $k=-1$. Therefore
$$
\ell\bigl(\Ecal xt^1(\Fcal,\Ocal)\bigr)
\equiv
h^1(\PBbb^2,\underline{\Fcal}(-1))
\pmod{2}.
$$
Computing using the Riemann--Roch formula, we have
$$
h^1(\PBbb^2,\underline{\Fcal}(-1))
\equiv c_2(\underline{\Fcal}) \pmod{2}.
$$
Since $c_1(\underline{\Fcal})\equiv 1\pmod{2}$ in this case, this gives
$$
\ell\bigl(\Ecal xt^1(\Fcal,\Ocal)\bigr)
\equiv
c_1(\underline{\Fcal})c_2(\underline{\Fcal})
\pmod{2}.
$$
The proof is complete.
\end{proof}

In particular, we obtain the following topological consequence.

\begin{cor}\label{Corollary-3.13}
For any smooth complex rank-two vector bundle $\underline F$ on $\PBbb^2$, let $F$ be its natural pullback to $S^5$. Then
$$
\delta(F)\equiv c_1(\underline F)c_2(\underline F)\pmod{2}.
$$
\end{cor}

\begin{proof}
By Schwarzenberger's theorem \cite[Theorem~6.2.1]{OSS11}, $\underline F$ admits a holomorphic structure. Denote the resulting holomorphic vector bundle by $\underline{\Fcal}$. The claim follows from Corollary \ref{Proposition-Z2 invariants for homogenous singularity}.
\end{proof}

As a consequence, we obtain Corollary \ref{Computation}.

\begin{proof}[Proof of Corollary \ref{Computation}]
Let $\widehat{\Fcal}$ be a locally free extension of $\Fcal$ on $\widehat B=\mathrm{Bl}_xB$, and let $D\cong\PBbb^2$ be the exceptional divisor. Let $\Fcal^c$ be the homogeneous singularity whose link is $\widehat{\Fcal}|_D$. Then $\Fcal$ and $\Fcal^c$ have the same underlying bundle on a small sphere around $x$, so
$$
\delta(\Fcal,x)=\delta(\Fcal^c,x).
$$
By Corollary \ref{Corollary-3.13},
$$
\delta(\Fcal^c,x)\equiv c_1(\widehat{\Fcal}|_D)c_2(\widehat{\Fcal}|_D)\pmod{2}.
$$
Applying Theorem \ref{Theorem-2} on a small ball around $x$ gives
$$
c_3(\Fcal,x)\equiv \delta(\Fcal,x)\pmod{2}.
$$
Combining these congruences gives
$$
c_3(\Fcal,x)\equiv c_1(\widehat{\Fcal}|_D)c_2(\widehat{\Fcal}|_D)\pmod{2},
$$
as claimed.
\end{proof}

\section{\texorpdfstring{Topological interpretations via relative $K$-groups}{Topological interpretations via relative K-groups}}\label{Section-K group}

For our purposes, we recall the relative $K$-theory constructions needed in the local setting and explain how the local third Chern class enters this framework.

\subsection{\texorpdfstring{The relevant relative $K$-groups}{The relevant relative K-groups}}

\subsubsection{\texorpdfstring{The topological relative $K$-group}{The topological relative K-group}}

Fix $0<r<1$ and set $S_r^5=\partial B(r)$. We use the standard identification
$$
K^0(\overline{B(r)},S_r^5)\cong \widetilde K^0(S^6),
$$
where $\overline{B(r)}/S_r^5\simeq S^6$; see \cite{Atiyah:2018} for background.

We recall how a finite complex exact on the boundary defines a relative $K$-class. Let
$$
\mathsf C^\bullet:\quad
0\to \mathsf C^0\xrightarrow{\partial_0}\mathsf C^1\xrightarrow{\partial_1}\cdots
\xrightarrow{\partial_{N-1}}\mathsf C^N\to 0
$$
be a finite complex of locally free sheaves on a neighborhood of $\overline{B(r)}$, exact on $\overline{B(r)}\setminus\{0\}$. Set
$$
\mathsf C^{\mathrm{ev}}:=\bigoplus_i\mathsf C^{2i},\qquad
\mathsf C^{\mathrm{odd}}:=\bigoplus_i\mathsf C^{2i+1}.
$$
After choosing Hermitian metrics on the $\mathsf C^i$, let $\partial_{\mathsf C}=\oplus_i\partial_i$ be the total differential and let $\partial_{\mathsf C}^*$ be its fiberwise adjoint. Then
$$
\mathcal D_{\mathsf C}:=\partial_{\mathsf C}+\partial_{\mathsf C}^*
$$
interchanges $\mathsf C^{\mathrm{ev}}$ and $\mathsf C^{\mathrm{odd}}$. Since $\mathsf C^\bullet$ is exact on $S_r^5$, the map
$$
\mathcal D_{\mathsf C}^+|_{S_r^5}:\mathsf C^{\mathrm{ev}}|_{S_r^5}\longrightarrow \mathsf C^{\mathrm{odd}}|_{S_r^5}
$$
is an isomorphism. Indeed, the kernel of $\mathcal D_{\mathsf C}$ on $S_r^5$ identifies with the cohomology of $\mathsf C^\bullet|_{S_r^5}$, which is zero by exactness. Therefore
$$
\beta_{\mathsf C^\bullet}:=
\left[\mathsf C^{\mathrm{ev}},\mathsf C^{\mathrm{odd}},\mathcal D_{\mathsf C}^+|_{S_r^5}\right]
\in K^0(\overline{B(r)},S_r^5)\cong \widetilde K^0(S^6).
$$
This class is independent of the Hermitian metrics, since different choices give homotopic boundary isomorphisms.

Set $\mathbf z=(z_1,z_2,z_3)$. The Koszul complex of the regular sequence $\mathbf z$ is
$$
\mathsf K^\bullet_{\mathbf z}:\quad
0\to \bigwedge^3(\Ocal_B^{\oplus 3})
\xrightarrow{\partial_{\mathbf z}}
\bigwedge^2(\Ocal_B^{\oplus 3})
\xrightarrow{\partial_{\mathbf z}}
\Ocal_B^{\oplus 3}
\xrightarrow{\mathbf z}
\Ocal_B\to 0,
$$
where $\partial_{\mathbf z}$ denotes the usual 
contraction. We place the last factor $\Ocal_B$ in degree $0$, so that $\bigwedge^j(\Ocal_B^{\oplus 3})$ lies in degree $-j$. Restricting to $\overline{B(r)}$, this complex is exact on $\overline{B(r)}\setminus\{0\}$ and hence defines a relative $K$-class
$$
\beta_{\mathsf K^\bullet_{\mathbf z}}\in K^0(\overline{B(r)},S_r^5)\cong \widetilde K^0(S^6).
$$

\begin{defi}
With this convention, the Bott class $\beta\in \widetilde K^0(S^6)$ is defined by
$$
\beta:=\beta_{\mathsf K^\bullet_{\mathbf z}}.
$$
\end{defi}

We use the following normalization of the Bott class; see
\cite[Section~2.6, pp.~98--100]{Atiyah:2018}.

\begin{prop}\label{BottGenerator}
The Bott class $\beta=\beta_{\mathsf K^\bullet_{\mathbf z}}$ generates $\widetilde K^0(S^6)$. Equivalently, it determines an isomorphism
$$
\widetilde K^0(S^6)\cong \ZBbb
$$
sending $\beta$ to $1$.
\end{prop}

Given this normalization, we make the following definition.

\begin{defi}
Let $\mathsf C^\bullet$ be a finite complex of locally free sheaves on a neighborhood of $\overline{B(r)}$, exact on $\overline{B(r)}\setminus\{0\}$. If
$$
\beta_{\mathsf C^\bullet}=\kappa\beta,
$$
then $\kappa$ is the relative $K$-theoretic charge of $\mathsf C^\bullet$.
\end{defi}

\subsubsection{The algebraic relative $K$-group and the Grothendieck group}

We recall the algebraic $K$-theoretic groups used below. Let $R=\Ocal_{\CBbb^3,0}$ and let $\mathfrak m=(z_1,z_2,z_3)$. We write $K^0_{\mathfrak m}(R)$ for the Grothendieck group of bounded complexes of finite free $R$-modules exact on $\Spec R\setminus\{\mathfrak m\}$, and $G_0^{\mathfrak m}(R)$ for the Grothendieck group of finite-length $R$-modules. See \cite[Chapter~V, Section~7.6]{Weibel:13} for $K$-theory with support and \cite[Chapter~II, Section~6]{Weibel:13} for $G_0$ of an abelian category.

We will use the following elementary $K_0$ fact. For the ingredients, see
\cite[Chapter~II, Proposition~6.6, Exercises~6.1, 6.3, 6.5, and Theorem~7.6]{Weibel:13}.

\begin{prop}\label{Prop-Algebraic-K-Groups}
Let $R=\Ocal_{\CBbb^3,0}$ and let $\mathfrak m=(z_1,z_2,z_3)$. Then
$$
K^0_{\mathfrak m}(R)\cong G^{\mathfrak m}_0(R)\cong \ZBbb.
$$
Under this identification, the Koszul complex $\mathsf K^\bullet_{\mathbf z}$ maps to the class $[R/\mathfrak m]$, hence to $1$.
\end{prop}

\begin{proof}
Since $R$ is regular local, every finite-length $R$-module has a finite free resolution. Thus every finite-length module defines a class in $K^0_{\mathfrak m}(R)$. Conversely, if $\mathsf C^\bullet$ is a bounded complex of finite free $R$-modules exact on $\Spec R\setminus\{\mathfrak m\}$, then each $H^i(\mathsf C^\bullet)$ has finite length. The Euler characteristic map
$$
[\mathsf C^\bullet]\mapsto \sum_i(-1)^i[H^i(\mathsf C^\bullet)]
$$
defines a homomorphism
$$
K^0_{\mathfrak m}(R)\to G^{\mathfrak m}_0(R).
$$
The inverse sends a finite-length module to the class of any finite free resolution. These two maps are inverse by the resolution theorem.

The category of finite-length $R$-modules has the unique simple object $R/\mathfrak m$. Hence, by dévissage, we know
$$
G^{\mathfrak m}_0(R)\cong \ZBbb,\qquad [M]\mapsto \ell_R(M),
$$
and sends $[R/\mathfrak m]$ to $1$. Finally, since $\mathsf K^\bullet_{\mathbf z}$ is a finite free resolution of $R/\mathfrak m$, its class maps to $[R/\mathfrak m]$, hence to $1$.
\end{proof}

The topological and algebraic constructions are compatible in the following sense.

\begin{prop}\label{Prop-Top-Alg-Comparison}
With the normalizations above, the comparison maps
$$
\alpha:K^0_{\mathfrak m}(R)\longrightarrow \widetilde K^0(S^6),
\qquad
\chi:K^0_{\mathfrak m}(R)\longrightarrow G_0^{\mathfrak m}(R)
$$
are isomorphisms. Moreover,
$$
\alpha\bigl([\mathsf K^\bullet_{\mathbf z}]\bigr)=\beta,
\qquad
\chi\bigl([\mathsf K^\bullet_{\mathbf z}]\bigr)=[R/\mathfrak m].
$$
\end{prop}

\begin{proof}
Let
$$
\alpha:K^0_{\mathfrak m}(R)\longrightarrow K^0(\overline{B(r)},S_r^5)
\cong \widetilde K^0(S^6)
$$
be defined as follows. Represent a class by a bounded complex $\mathsf C^\bullet$ of finite free $R$-modules, exact on $\Spec R\setminus\{\mathfrak m\}$. After decreasing $r$ if necessary, choose holomorphic representatives of the differentials on $B(r)$. The resulting complex is exact on $S_r^5$, so it defines the relative class $\beta_{\mathsf C^\bullet}$. By construction,
$$
\alpha\bigl([\mathsf K^\bullet_{\mathbf z}]\bigr)=\beta.
$$
By Propositions \ref{BottGenerator} and \ref{Prop-Algebraic-K-Groups}, both sides are generated by these classes. Hence $\alpha$ is an isomorphism.

On the other hand, the Euler characteristic map
$$
\chi([\mathsf C^\bullet])=\sum_i(-1)^i[H^i(\mathsf C^\bullet)]
$$
identifies $K^0_{\mathfrak m}(R)$ with $G_0^{\mathfrak m}(R)$. Under this identification, $\mathsf K^\bullet_{\mathbf z}$ maps to its only nonzero cohomology module, namely $R/\mathfrak m$. This proves the claim.
\end{proof}

\subsection{Constructions arising from rank-two torsion-free sheaves}

Let $\Fcal$ be a rank-two torsion-free sheaf on $B$ with an isolated point singularity at $0$. Choose a local resolution
$$
0\to \Ocal_B^a
\xrightarrow{\Phi_1}
\Ocal_B^b
\xrightarrow{\Phi_2}
\Ocal_B^{b-a+2}
\to \Fcal\to 0.
$$
After shrinking $B$, fix a trivialization
$$
\sigma:\det(\Fcal)\xrightarrow{\sim}\Ocal_B.
$$
The determinant pairing for the rank-two reflexive sheaf $\Fcal^{**}$, together with $\sigma$, gives an isomorphism
$$
\theta:\Fcal^{**}\xrightarrow{\sim}(\Fcal^{**})^*\cong \Fcal^*.
$$
Composing the natural map $\Fcal\to \Fcal^{**}$ with $\theta$, we obtain a morphism
$$
j:\Fcal\to \Fcal^*.
$$
This morphism is an isomorphism on $B\setminus\{0\}$.

We use $j$ to construct a self-dual complex. Let
$$
\Psi:\Ocal_B^{b-a+2}\to (\Ocal_B^{b-a+2})^*
$$
be the composition
$$
\Ocal_B^{b-a+2}\to \Fcal
\xrightarrow{j}
\Fcal^*
\hookrightarrow
(\Ocal_B^{b-a+2})^*.
$$
Then we obtain the finite complex
$$
\mathsf C^\bullet_{\Fcal}:\quad
0\to \Ocal_B^a
\xrightarrow{\Phi_1}
\Ocal_B^b
\xrightarrow{\Phi_2}
\Ocal_B^{b-a+2}
\xrightarrow{\Psi}
(\Ocal_B^{b-a+2})^*
\xrightarrow{\Phi_2^*}
(\Ocal_B^b)^*
\xrightarrow{\Phi_1^*}
(\Ocal_B^a)^*
\to 0,
$$
where $\Ocal_B^a$ is placed in degree $0$.

\begin{lem}\label{Lemma-Exact-Away-Origin}
The complex $\mathsf C^\bullet_{\Fcal}$ is exact on $B\setminus\{0\}$.
\end{lem}

\begin{proof}
On $B\setminus\{0\}$, the sheaf $\Fcal$ is locally free and $j:\Fcal\to \Fcal^*$ is an isomorphism. Hence $\mathsf C^\bullet_{\Fcal}$ is obtained by splicing a locally free resolution of $\Fcal$ with its dual resolution along $j$, and is therefore exact.
\end{proof}

Equip the trivial bundles in $\mathsf C^\bullet_{\Fcal}$ with their standard flat Hermitian metrics. Let $\partial_{\Fcal}$ be the total differential and set
$$
E^{+}:=\Ocal_B^a\oplus \Ocal_B^{b-a+2}\oplus(\Ocal_B^b)^*,
\qquad
E^{-}:=\Ocal_B^b\oplus(\Ocal_B^{b-a+2})^*\oplus(\Ocal_B^a)^*.
$$
On $B\setminus\{0\}$, the operator
$$
\mathcal D_{\Fcal}:=\partial_{\Fcal}+\partial_{\Fcal}^*:E^+\to E^-
$$
is an isomorphism. Thus $\mathcal D_{\Fcal}|_{S^5}$ defines a clutching map, and hence a vector bundle $V_{\Fcal}$ of rank $2b+2$ over $S^6$. We define
$$
\beta_{\Fcal}:=[V_{\Fcal}]-[\underline{\CBbb}^{\,2b+2}]
\in \widetilde K^0(S^6).
$$

\begin{lem}\label{Lemma-Beta-Resolution-Independent}
The class $\beta_{\Fcal}$ is independent of the choice of local resolution. In particular, it is an intrinsic topological invariant of the germ of $\Fcal$ at $0$.
\end{lem}

\begin{proof}
The complex $\mathsf C^\bullet_{\Fcal}$ defines a class in the relative $K$-group with support at $0$. Replacing the chosen local resolution by another one changes the resulting perfect complex by direct sum of free factors; see \cite[Theorem~20.2]{Eisenbud:13}. Hence it does not change its class in the relative $K$-group. Under the clutching identification with $\widetilde K^0(S^6)$, the corresponding class is unchanged.
\end{proof}

The same complex also defines an algebraic relative $K$-class
$$
[\mathsf C^\bullet_{\Fcal}]\in K^0_{\mathfrak m}(R),
$$
where $R=\Ocal_{B,0}\cong \Ocal_{\CBbb^3,0}$ and $\mathfrak m=(z_1,z_2,z_3)$. Since $\mathsf C^\bullet_{\Fcal}$ is exact away from $0$, its cohomology sheaves have finite length. We write
$$
[H^\bullet(\mathsf C^\bullet_{\Fcal})]
:=
\sum_i(-1)^i[H^i(\mathsf C^\bullet_{\Fcal})]
\in G_0^{\mathfrak m}(R).
$$

\begin{lem}\label{Lemma-Homology-Class}
In $G_0^{\mathfrak m}(R)$,
$$
[H^\bullet(\mathsf C^\bullet_{\Fcal})]
=
[\Ecal xt^1(\Fcal,\Ocal_B)]-[\Ecal xt^2(\Fcal,\Ocal_B)]-[\Fcal^{**}/\Fcal].
$$
Equivalently,
$$
[H^\bullet(\mathsf C^\bullet_{\Fcal})]
=
[\Ecal xt^1(\Fcal,\Ocal_B)]-2[\Fcal^{**}/\Fcal].
$$
\end{lem}

\begin{proof}
Let $\Qcal=\Fcal^{**}/\Fcal$. By construction, the only nonzero cohomology sheaves of $\mathsf C^\bullet_{\Fcal}$ are
$$
H^3(\mathsf C^\bullet_{\Fcal})\cong \Qcal,\qquad
H^4(\mathsf C^\bullet_{\Fcal})\cong \Ecal xt^1(\Fcal,\Ocal_B),\qquad
H^5(\mathsf C^\bullet_{\Fcal})\cong \Ecal xt^2(\Fcal,\Ocal_B).
$$
Therefore
$$
[H^\bullet(\mathsf C^\bullet_{\Fcal})]
=
-[\Qcal]+[\Ecal xt^1(\Fcal,\Ocal_B)]-[\Ecal xt^2(\Fcal,\Ocal_B)].
$$
Since $\Fcal^{**}$ is reflexive on the smooth threefold $B$, $\Ecal xt^i(\Fcal^{**},\Ocal_B)=0$ for $i\geq 2$. Hence by Lemma \ref{Lemma-I1},
$$
[H^\bullet(\mathsf C^\bullet_{\Fcal})]
=
[\Ecal xt^1(\Fcal,\Ocal_B)]-2[\Qcal],
$$
as claimed.
\end{proof}

We now obtain the desired $K$-theoretic interpretation of the local third Chern class.

\begin{thm}\label{Thm-K charge interpretation}
Suppose $\Fcal$ is a rank-two torsion-free sheaf on $B$ with an isolated point singularity at $0$. Then
$$
\beta_{\Fcal}=c_3(\Fcal,0)\beta.
$$
In particular, $\kappa_{\Fcal}=c_3(\Fcal,0)$.
\end{thm}

\begin{proof}
By Lemma \ref{Lemma-Homology-Class} and the definition of the local third Chern class,
$$
\chi([\mathsf C^\bullet_{\Fcal}])
=
[H^\bullet(\mathsf C^\bullet_{\Fcal})]
=
c_3(\Fcal,0)[R/\mathfrak m]
=
c_3(\Fcal,0)\chi([\mathsf K^\bullet_{\mathbf z}])
$$
in $G_0^{\mathfrak m}(R)$. Since $\chi$ is an isomorphism,
$$
[\mathsf C^\bullet_{\Fcal}]
=
c_3(\Fcal,0)[\mathsf K^\bullet_{\mathbf z}]
$$
in $K^0_{\mathfrak m}(R)$. Applying the comparison map $\alpha$ gives
$$
\beta_{\Fcal}
=
\alpha([\mathsf C^\bullet_{\Fcal}])
=
c_3(\Fcal,0)\alpha([\mathsf K^\bullet_{\mathbf z}])
=
c_3(\Fcal,0)\beta.
$$
\end{proof}

\begin{cor}
Let $\Fcal$ be a rank-two torsion-free sheaf on a ball $B\subset \CBbb^3$ with
$$
\sing(\Fcal)=\{x_1,\ldots,x_N\}.
$$
Then the boundary class defined by the self-dual complex satisfies
$$
\beta(\Fcal)
=
\left(\sum_{i=1}^N c_3(\Fcal,x_i)\right)\beta.
$$
\end{cor}

\begin{cor}
Let $\Fcal$ be a rank-two torsion-free sheaf with finitely many singularities. Suppose
$$
\beta(\Fcal)=k\beta.
$$
Then
$$
k\equiv \delta(\Fcal) \pmod{2}.
$$
\end{cor}

\begin{cor}
Let $\Ecal$ be a family of rank-two torsion-free sheaves on $B\times \Delta$. Suppose that each fiber $\Ecal_t$ has finitely many isolated point singularities, all contained in $B(1/2)$. Then
$$
\kappa_{\Ecal}:\Delta\to \widetilde K^0(S^6),\qquad t\mapsto \beta(\Ecal_t),
$$
is constant. More precisely, for every $t\in\Delta$,
$$
\kappa_{\Ecal}(t)
=
\left(\sum_{x\in \sing(\Ecal_t)}c_3(\Ecal_t,x)\right)\beta
=
\left(\rank(p_*\Ecal xt^1(\Ecal^{**},\Ocal_X))-2\rank(p_*(\Ecal^{**}/\Ecal))\right)\beta.
$$
\end{cor}

\section{Simple and classical interpretations in the reflexive case}\label{Section-Reflexive}
The previous section gave a $K$-theoretic interpretation of the local third Chern class for rank-two torsion-free sheaves with isolated point singularities. We now specialize to rank-two reflexive sheaves with a single isolated point singularity.

\subsection{A notion of degree}

Suppose $\Fcal$ is a rank-two reflexive sheaf on $B$ with an isolated point singularity at $0$. After shrinking $B$, choose a presentation
$$
0\to \Ocal_B^a \xrightarrow{\Phi} \Ocal_B^{a+2}\to \Fcal\to 0.
$$
Let
$$
M=\operatorname{Hom}(\CBbb^a,\CBbb^{a+2}),\qquad
\Sigma=\{A\in M:\rank A\leq a-1\}.
$$
Denote by $\Sigma_{\mathrm{sm}}$ the smooth locus of $\Sigma$. We regard $\Phi$ as a holomorphic map $\Phi:B\to M$. Since $\Fcal$ is locally free on $B\setminus\{0\}$, the restriction of $\Phi$ maps $B\setminus\{0\}$ to $M\setminus\Sigma$.

We will use the following elementary fact about $M\setminus\Sigma$.

\begin{lem}\label{Lemma-Cohomology-Minus-Sigma}
There is an isomorphism
$$
H^5(M\setminus\Sigma,\ZBbb)\cong \ZBbb.
$$
Moreover, a generator may be chosen so that its restriction under
$S_A^5\hookrightarrow M\setminus\Sigma$ is the standard generator of
$H^5(S_A^5,\ZBbb)$, where $A$ is any smooth point of $\Sigma$ and $S_A^5$ is the link of a transversal slice to $\Sigma$ in $M$ at $A$.
\end{lem}

Since $\Phi$ restricts to a map
$$
\Phi:B^*:=B\setminus\{0\}\to M\setminus \Sigma,
$$
we can define a notion of degree for $\Phi$ as follows.

\begin{defi}
The degree of $\Phi$ at $0$, denoted by $\deg(\Phi,0)$, is the integer for which the following diagram commutes:
$$
\begin{tikzcd}
H^5(M\setminus \Sigma,\ZBbb) \arrow{r}{\Phi^*} \arrow{d}[swap]{\lambda_A^*}
& H^5(B^*,\ZBbb) \arrow{d}{\lambda_0^*} \\
H^5(S_A^5,\ZBbb) \arrow{r}{\deg(\Phi,0)}
& H^5(S^5,\ZBbb).
\end{tikzcd}
$$
Here $A\in\Sigma_{\mathrm{sm}}$, $S_A^5$ is the link of a transversal slice to $\Sigma$ in $M$ at $A$, and $\lambda_A:S_A^5\hookrightarrow M\setminus\Sigma$ and $\lambda_0:S^5\hookrightarrow B^*$ are the natural inclusions.
\end{defi}

\begin{exam}
For a simple singularity, the presentation map is essentially the inclusion
$$
\Phi:B\to \CBbb^3,
$$
where $\Sigma=\{0\}$. In this case,
$$
\deg(\Phi,0)=1.
$$
\end{exam}

\begin{prop}
Suppose $\Fcal$ is a rank-two reflexive sheaf on $B$ with an isolated point singularity at $0$, and let $\Phi$ be a presentation map as above. Then
$$
\deg(\Phi,0)=c_3(\Fcal,0).
$$
In particular, $\deg(\Phi,0)$ is independent of the choice of $\Phi$ and depends only on the germ of $\Fcal$ at $0$.
\end{prop}

\begin{proof}
Let $\Phi_\bullet:B\times\Delta\to M$ be the perturbing family constructed in Proposition \ref{Perturbation}, with $\Phi_0=\Phi$, and let $\Ecal$ be the induced family
$$
0\to \Ocal_X^a\xrightarrow{\Phi_\bullet}\Ocal_X^{a+2}\to \Ecal\to 0,
\qquad X=B\times\Delta.
$$
After shrinking $\Delta$, its central fiber satisfies $\Ccal\cong\Fcal$, and for every $0<|t|\ll1$ the sheaf $\Ecal_t$ has only simple singularities, all contained in $B(1/2)$.

The maps $\Phi_t|_{\partial B}$ give a homotopy in $M\setminus\Sigma$. By homotopy invariance of the boundary degree and additivity over isolated singularities,
$$
\deg(\Phi,0)=\sum_{x\in\sing(\Ecal_t)}\deg(\Phi_t,x).
$$
Each singularity of $\Ecal_t$ is simple, hence
$$
\deg(\Phi_t,x)=1=c_3(\Ecal_t,x).
$$
Therefore
$$
\deg(\Phi,0)=\sum_{x\in\sing(\Ecal_t)}c_3(\Ecal_t,x)=c_3(\Ecal_t).
$$
Applying Theorem \ref{Theorem-1} to the family $\Ecal$ gives
$$
c_3(\Ecal_t)=c_3(\Ccal).
$$
Since $\Ccal\cong\Fcal$ and $\Fcal$ has only the singularity $0$,
$$
c_3(\Ccal)=c_3(\Fcal,0).
$$
Combining the equalities gives
$$
\deg(\Phi,0)=c_3(\Fcal,0).
$$
\end{proof}
The preceding proposition makes the following definition intrinsic.

\begin{defi}
Let $\Fcal$ be a rank-two reflexive sheaf on $B$ with an isolated point singularity at $0$. Choose a local presentation near $0$ of the form
$$
0\to \Ocal_B^a \xrightarrow{\Phi} \Ocal_B^{a+2}\to \Fcal\to 0.
$$
The degree of $\Fcal$ at $0$ is
$$
\deg(\Fcal,0):=\deg(\Phi,0).
$$
\end{defi}

As an immediate consequence, the local $K$-theoretic charge is computed by this degree.

\begin{cor}
Let $\Fcal$ be a rank-two reflexive sheaf on $B$ with an isolated point singularity at $0$. Then
$$
\beta(\Fcal,0)=\deg(\Fcal,0)\beta.
$$
In particular,
$$
\kappa_{\Fcal}(0)=c_3(\Fcal,0)=\deg(\Fcal,0).
$$
\end{cor}

\subsection{Relation to the Fitting scheme and the Buchsbaum--Rim multiplicity}
In this section, we use the topological interpretation of the local third Chern class to relate it to other algebraic invariants of the sheaf, such as the Fitting ideal and the Buchsbaum--Rim multiplicity.

Below we always assume $\Fcal$ is a rank-two reflexive sheaf with an isolated point singularity at $0$. Take a minimal resolution of $\Fcal$ at $0$ as 
$$
0\to \Ocal^{a} \xrightarrow{\Phi} \Ocal^{a+2} \to \Fcal \to 0.
$$
There are two closely related objects.

\begin{enumerate}
\item The Fitting ideal defining the non-locally-free locus is
$$
I_{\Fcal}:=\operatorname{Fitt}_2(\Fcal)
=I_a(\Phi)
=
\operatorname{Im}\left(
\bigwedge^a\Phi^*:
\bigwedge^a(\Ocal^{a+2})^*
\to
\bigwedge^a(\Ocal^a)^*
\cong
\Ocal
\right).
$$
Equivalently, $I_{\Fcal}$ is the ideal generated by the $a\times a$
minors of a matrix for $\Phi$. Since the singularity is isolated at $0$,
this ideal is $\mathfrak m_0$-primary where $\mathfrak m_0$ denotes the maximal ideal of the origin and defines the zero-dimensional
Fitting scheme
$$
\operatorname{Spec}(\Ocal/I_{\Fcal}).
$$

\item Dualizing the resolution gives
$$
(\Ocal^{a+2})^*
\xrightarrow{\Phi^*}
(\Ocal^a)^*
\to
\Ecal xt^1(\Fcal,\Ocal)
\to 0.
$$
Hence
$$
\Ecal xt^1(\Fcal,\Ocal)\cong \operatorname{coker}(\Phi^*),
\qquad
\operatorname{Fitt}_0\bigl(\Ecal xt^1(\Fcal,\Ocal)\bigr)=I_{\Fcal}.
$$
\end{enumerate}

The following elementary fact is well-known.
\begin{lem}
$\Ecal xt^1(\Fcal,\Ocal)$ is naturally an
$\Ocal/I_{\Fcal}$-module, and hence is supported on the Fitting scheme
defined by $I_{\Fcal}$.
\end{lem}

\begin{proof}
It suffices to show that every $a\times a$ minor $\Delta$ of $\Phi$ annihilates $\Ccal oker(\Phi^*)$. Choose the corresponding $a\times a$ submatrix $A$ of $\Phi^*$. Since $A\,\operatorname{adj}(A)=\Delta\operatorname{id}$, for every local section $v$ of $(\Ocal^a)^*$ we get $\Delta v\in\operatorname{Im}(\Phi^*)$. This proves the claim.
\end{proof}

We need the following tautological interpretation of the length of the fitting ideal.
\begin{lem}
    $\ell(\Ocal/\Ical_{\Fcal},0)$ is equal to the local intersection multiplicity of the graph of $\Phi: B \to M$ with $B\times \Sigma$ in $B\times M$ at $(0,\Phi(0))$. 
\end{lem}

The preceding interpretation gives the following identity.

\begin{prop}\label{Prop-Fitting-Length-Degree}
$
c_3(\Fcal,0)=\deg(\Phi,0)=\ell(\Ocal_B/\Ical_{\Fcal},0).$
\end{prop}

\begin{proof}
By the degree comparison above,
$$
c_3(\Fcal,0)=\deg(\Phi,0).
$$
It remains to identify $\deg(\Phi,0)$ with the Fitting length. By the previous lemma, $\ell(\Ocal_B/\Ical_{\Fcal},0)$ is the local intersection multiplicity of $\Gamma_\Phi$ and $B\times\Sigma$ at $(0,\Phi(0))$. After a generic small perturbation $\Phi_t$ and after shrinking to a small ball $B(\epsilon)$ around $0$, this local intersection becomes transverse, and each point contributes multiplicity one. Hence
$$
\ell(\Ocal_B/\Ical_{\Fcal},0)
=
\#\bigl(\Phi_t^{-1}(\Sigma)\cap B(\epsilon)\bigr)
$$
while the right hand side is $\deg(\Phi,0)$. Therefore
$$
\deg(\Phi,0)=\ell(\Ocal_B/\Ical_{\Fcal},0),
$$
and the claim follows.
\end{proof}

The interpretation above also connects $c_3(\Fcal,0)$ with two classical multiplicities. First consider the local intersection ring. Set
$$
R:=\Ocal_{B\times M,(0,\Phi(0))}.
$$
Let $\Gamma_\Phi\subset B\times M$ be the graph of $\Phi$, and define
$$
\Ical_\Phi:=I_{\Gamma_\Phi}+I_{B\times\Sigma}\subset R.
$$
Thus
$$
A_\Phi:=R/\Ical_\Phi
$$
is the local ring of the scheme-theoretic intersection $\Gamma_\Phi\cap(B\times\Sigma)$ at $(0,\Phi(0))$. Since $\Fcal$ has an isolated singularity at $0$, the ring $A_\Phi$ is Artinian. Moreover, the defining equations of $\Sigma$ are the maximal minors of the universal matrix, so restriction to the graph gives
$$
A_\Phi\cong \Ocal_{B,0}/I_{\Fcal}.
$$

Let $\mathfrak n_\Phi$ be the maximal ideal of $A_\Phi$. We define the Hilbert--Samuel intersection multiplicity by
$$
\operatorname{hs}(\Fcal,0):=e_{\mathfrak n_\Phi}(A_\Phi).
$$
Here recall that the Hilbert--Samuel function of $(A_\Phi,\mathfrak n_\Phi)$ is
$$
H_{\mathfrak n_\Phi,A_\Phi}(q):=\ell(A_\Phi/\mathfrak n_\Phi^q),
\qquad q\geq 1.
$$
For $q\gg0$, it agrees with a polynomial $P_{\mathfrak n_\Phi,A_\Phi}(q)$ whose leading term is
$$
\frac{e_{\mathfrak n_\Phi}(A_\Phi)}{(\dim A_\Phi)!}q^{\dim A_\Phi}.
$$
Since $A_\Phi$ is zero-dimensional, $\mathfrak n_\Phi$ is nilpotent. Hence, for $q\gg0$,
$$
A_\Phi/\mathfrak n_\Phi^q=A_\Phi,
$$
and therefore
$$
H_{\mathfrak n_\Phi,A_\Phi}(q)=\ell(A_\Phi).
$$
Thus $P_{\mathfrak n_\Phi,A_\Phi}$ is the constant polynomial $\ell(A_\Phi)$, and since $\dim A_\Phi=0$,
$$
e_{\mathfrak n_\Phi}(A_\Phi)=\ell(A_\Phi).
$$
Therefore
$$
\operatorname{hs}(\Fcal,0)
=
\ell(\Ocal_{B,0}/I_{\Fcal})
=
c_3(\Fcal,0).
$$

We now recall the Buchsbaum--Rim multiplicity in the present setting; see \cite[3.1, 3.4, 4.5(2)]{Buchsbaum&Rim:64}. Let
$$
R_0:=\Ocal_{B,0},\qquad E:=(R_0^a)^*.
$$
Dualizing the presentation of $\Fcal$ gives
$$
(R_0^{a+2})^*\xrightarrow{\Phi^*}E\to \Ecal xt^1(\Fcal,\Ocal_B)_0\to 0.
$$
Set
$$
N:=\operatorname{Im}(\Phi^*)\subset E.
$$
Then
$$
E/N\cong \Ecal xt^1(\Fcal,\Ocal_B)_0,
$$
so $E/N$ has finite length.

For $q\geq0$, define
$$
\mathcal R_q(N):=
\operatorname{Im}\left(\operatorname{Sym}_{R_0}^q N
\to
\operatorname{Sym}_{R_0}^q E\right).
$$
Equivalently,
$$
\mathcal R(N):=\bigoplus_{q\geq0}\mathcal R_q(N)
\subset \operatorname{Sym}_{R_0}(E)
$$
is the Rees algebra of $N$ inside $\operatorname{Sym}_{R_0}(E)$. The Buchsbaum--Rim function is
$$
H_{N,E}(q):=
\ell_{R_0}\left(\operatorname{Sym}_{R_0}^q E/\mathcal R_q(N)\right).
$$
Since $\dim R_0=3$ and $\rank_{R_0}E=a$, for $q\gg0$ this agrees with a polynomial $P_{N,E}(q)$ of degree
$$
3+a-1=a+2.
$$
Its leading term is
$$
\frac{e_{\mathrm{BR}}(N,E)}{(a+2)!}q^{a+2}.
$$
The integer $e_{\mathrm{BR}}(N,E)$ is the Buchsbaum--Rim multiplicity of $N\subset E$. In our notation, set
$$
\operatorname{br}(\Fcal,0):=e_{\mathrm{BR}}(N,E).
$$

We now explain why this multiplicity reduces to a length in our case. By a parameter module in a free module $E$ over a local ring $R_0$ we mean a submodule of finite colength generated by
$$
\dim R_0+\rank_{R_0}E-1
$$
elements. Thus $N$ is a parameter module: the quotient $E/N$ has finite length, and $N$ is generated by the $a+2$ columns of
$$
\Phi^*:(R_0^{a+2})^*\to E,
$$
while
$$
a+2=\dim R_0+\rank_{R_0}E-1.
$$
Since $R_0$ is regular local, hence Cohen--Macaulay, the Buchsbaum--Rim formula for parameter modules gives
$$
\operatorname{br}(\Fcal,0)
=
e_{\mathrm{BR}}(N,E)
=
\ell_{R_0}(E/N).
$$
Using $E/N\cong \Ecal xt^1(\Fcal,\Ocal_B)_0$, we obtain
$$
\operatorname{br}(\Fcal,0)
=
\ell(\Ecal xt^1(\Fcal,\Ocal_B),0).
$$
Since $\Fcal$ is reflexive,
$$
c_3(\Fcal,0)=\ell(\Ecal xt^1(\Fcal,\Ocal_B),0).
$$
Therefore
$$
\operatorname{br}(\Fcal,0)=c_3(\Fcal,0).
$$

Combining these identities gives the following corollary.

\begin{cor}
Let $\Fcal$ be a rank-two reflexive sheaf on $B$ with an isolated point singularity at $0$. Then the following quantities are equal:
\begin{enumerate}
\item the local third Chern class $c_3(\Fcal,0)$;
\item the local $K$-theoretic charge $\kappa_{\Fcal}(0)$;
\item the length $\ell(\Ecal xt^1(\Fcal,\Ocal_B),0)$;
\item the topological degree $\deg(\Fcal,0)$;
\item the Hilbert--Samuel multiplicity $\operatorname{hs}(\Fcal,0)$;
\item the length of the Fitting scheme $\ell(\Ocal_B/\Ical_{\Fcal},0)$;
\item the Buchsbaum--Rim multiplicity $\operatorname{br}(\Fcal,0)$.
\end{enumerate}
\end{cor}

Now we give an example illustrating the equalities above.

\begin{exam}
Consider the sheaf $\Fcal$ defined by the presentation
$$
0\to \Ocal_B
\xrightarrow{\Phi}
\Ocal_B^3\to \Fcal\to 0,
$$
where
$$
\Phi(z)=
\begin{pmatrix}
z_1\\
z_2\\
z_3^k
\end{pmatrix}.
$$
Since $(z_1,z_2,z_3^k)$ is a regular sequence and its common zero locus is $0$, the sheaf $\Fcal$ is a rank-two reflexive sheaf with an isolated singularity at $0$.

Let $R_0=\Ocal_{B,0}$. Dualizing the presentation gives
$$
\Phi^*:(R_0^3)^*\to R_0,
\qquad
\Phi^*(a,b,c)=az_1+bz_2+cz_3^k.
$$
Thus
$$
\Ecal xt^1(\Fcal,\Ocal_B)_0\cong \Ccal oker(\Phi^*)
\cong R_0/(z_1,z_2,z_3^k).
$$
In particular,
$$
\ell(\Ecal xt^1(\Fcal,\Ocal_B),0)=k.
$$
Since the presentation has rank one on the left, the Fitting ideal is the ideal of the entries of $\Phi$:
$$
\Ical_{\Fcal}=I_1(\Phi)=(z_1,z_2,z_3^k).
$$
Hence
$$
\ell(\Ocal_B/\Ical_{\Fcal},0)=k.
$$

We now compute the Buchsbaum--Rim multiplicity directly. In this example,
$$
E:=(R_0)^*\cong R_0,\qquad
N:=\operatorname{Im}(\Phi^*)=(z_1,z_2,z_3^k)\subset E.
$$
Therefore
$$
\operatorname{Sym}_{R_0}^n E\cong R_0,
\qquad
\mathcal R_n(N)=N^n=(z_1,z_2,z_3^k)^n.
$$
It follows that
$$
\operatorname{Sym}_{R_0}^nE/\mathcal R_n(N)
\cong R_0/(z_1,z_2,z_3^k)^n.
$$
We compute the length of this quotient. A monomial $z_1^\alpha z_2^\beta z_3^\gamma$ survives modulo $(z_1,z_2,z_3^k)^n$ precisely when
$$
\alpha+\beta+\left\lfloor \frac{\gamma}{k}\right\rfloor\leq n-1.
$$
Write
$$
\gamma=kq+s,\qquad 0\leq s\leq k-1.
$$
For fixed $q$, the number of pairs $(\alpha,\beta)$ satisfying
$$
\alpha+\beta\leq n-1-q
$$
is
$$
\binom{n-q+1}{2}.
$$
Since there are $k$ choices for $s$, we obtain
$$
\ell_{R_0}\bigl(R_0/(z_1,z_2,z_3^k)^n\bigr)
=
k\sum_{q=0}^{n-1}\binom{n-q+1}{2}
=
k\binom{n+2}{3}.
$$
Thus
$$
\ell_{R_0}\bigl(R_0/(z_1,z_2,z_3^k)^n\bigr)
=
\frac{k}{3!}n^3+O(n^2).
$$
Hence the normalized leading coefficient is $k$, and therefore
$$
\operatorname{br}(\Fcal,0)=k.
$$

Finally, consider the perturbation
$$
\Phi_t(z)=
\begin{pmatrix}
z_1\\
z_2\\
z_3^k-t
\end{pmatrix},
\qquad
0<|t|\ll 1.
$$
Its zero set near $0$ consists of the $k$ points
$$
(0,0,\zeta),\qquad \zeta^k=t.
$$
Each zero is simple. Therefore
$$
\deg(\Fcal,0)=k.
$$
Consequently, in this example,
$$
\begin{aligned}
\operatorname{br}(\Fcal,0)
&=
\deg(\Fcal,0)
=
\ell(\Ocal_B/\Ical_{\Fcal},0) \\
&=
\ell\bigl(\Ecal xt^1(\Fcal,\Ocal_B),0\bigr)
=
c_3(\Fcal,0)
=
k.
\end{aligned}
$$
\end{exam}
\bibliography{papers}
\end{document}